\documentclass[a4paper,11pt,reqno]{amsart}

\usepackage{a4wide}

\usepackage[utf8]{inputenc}

\PassOptionsToPackage{T1}{fontenc}
  \usepackage{fontenc}

\usepackage{amsmath}
\usepackage{amssymb}
\usepackage{amsthm}
\usepackage{fancyhdr} 
\usepackage{graphicx} 
\usepackage{enumerate}
\usepackage{cite}

\allowdisplaybreaks
\usepackage[toc,page]{appendix}
\usepackage{hyperref}    

\usepackage[all]{xy}
\usepackage{latexsym}
\usepackage{amscd}

\usepackage{xspace} 
\usepackage{calc} 
\usepackage{bm}

\usepackage{tikz-cd} 
\usepackage{tikz,tikz-3dplot}

\usepackage{float}

\usepackage{forest}

\definecolor{cof}{RGB}{219,144,71}
\definecolor{pur}{RGB}{186,146,162}
\definecolor{greeo}{RGB}{91,173,69}
\definecolor{greet}{RGB}{52,111,72}

\tdplotsetmaincoords{70}{110}

\newcommand{\R}{\mathbb R}
\newcommand{\Q}{\mathbb Q}
\newcommand{\Z}{\mathbb Z}
\newcommand{\N}{\mathbb N}
\newcommand{\C}{\mathbb C}
\newcommand{\CP}{\mathbb P}
\newcommand{\oCP}{\overline{\mathbb P}}
\newcommand{\T}{\mathbb T}

\newcommand{\Aa}{{\mathcal A}}
\newcommand{\Cc}{{\mathcal C}}
\newcommand{\Dd}{{\mathcal D}}
\newcommand{\Ff}{{\mathcal F}}

\newcommand{\Ss}{{\mathcal S}}
\newcommand{\Tt}{{\mathcal T}}
\newcommand{\Vv}{{\mathcal V}}
\newcommand{\Ee}{{\mathcal E}}
\newcommand{\Oo}{{\mathcal O}}
\newcommand{\Pp}{{\mathcal P}}
\newcommand{\Qq}{{\mathcal Q}}

\newcommand{\om}{\omega}

\DeclareMathOperator{\Ham}{Ham}
\DeclareMathOperator{\conv}{conv}
\DeclareMathOperator{\cone}{cone}
\DeclareMathOperator{\rank}{rank}

\DeclareMathOperator{\Ehr}{Ehr}
\DeclareMathOperator{\Pyr}{Pyr}
\DeclareMathOperator{\Bipyr}{Bipyr}
\DeclareMathOperator{\Pbipyr}{Pbipyr}

\newtheorem{theorem}{Theorem}[section]
\newtheorem{lemma}[theorem]{Lemma}
\newtheorem{thm}[theorem]{Theorem}
\newtheorem{proposition}[theorem]{Proposition}
\newtheorem{prop}[theorem]{Proposition}
\newtheorem{corollary}[theorem]{Corollary}

\newtheorem{definition}[theorem]{Definition}

\newtheorem{remark}[theorem]{Remark}
\newtheorem{rem}[theorem]{Remark}
\newtheorem{example}[theorem]{Example}

\begin{document}

\title[Contact rigidity/flexibility and Ehrhart theory of toric diagrams]
{Contact flexibility and rigidity for toric Gorenstein prequantizations and 
Ehrhart theory of toric diagrams}

\author[M.~Abreu]{Miguel Abreu}
\author[L.~Macarini]{Leonardo Macarini}
\author[A.~Rocha-Neves]{Ant\'{o}nio Rocha-Neves}

\address{Center for Mathematical Analysis, Geometry and Dynamical Systems,
Instituto Superior T\'ecnico, Universidade de Lisboa, 
Av. Rovisco Pais, 1049-001 Lisboa, Portugal}
\email{miguel.abreu@tecnico.ulisboa.pt, antonio.r.neves@tecnico.ulisboa.pt}

\address{IMPA, Estrada Dona Castorina, 110, Rio de Janeiro, 22460-320, Brazil}
\email{leonardo@impa.br}
  

\thanks{MA was partially supported by Funda\c c\~ao para a Ci\^encia e Tecnologia (FCT), Portugal, 
through grants 2023.13969.PEX and UID/4459/2025.  LM was partially supported by FAPERJ and CNPq/Brazil. 
ARN was partially supported by Funda\c c\~ao Calouste Gulbenkian.}

\begin{abstract}
Gorenstein toric contact manifolds are good toric contact manifolds with zero first Chern class that
are completely determined by certain integral convex polytopes called toric diagrams. The Ehrhart
polynomial of these toric diagrams determines and is determined by the contact Betti numbers of 
the corresponding contact manifolds, i.e. the dimension of their cylindrical contact homology in each
degree. In this paper we look into the following natural question: to what extent do these contact 
invariants determine the Gorenstein toric contact manifold? Flexibility is the norm and we illustrate 
it with the family of Gorenstein toric contact manifolds that arise as the prequantization 
of monotone iterated $\CP^1$-bundles, i.e. monotone Bott manifolds. In each dimension, the Ehrhart polynomial 
of their toric diagrams is equal to the Ehrhart polynomial of the cross-polytope, corresponding to the 
monotone prequantization of  $\CP^1 \times \cdots \times \CP^1$, and we describe the unimodular 
classification of these toric diagrams. On the rigidity side, we will show that the primitive prequantization 
of $\CP^1 \times \cdots \times \CP^1$ is rigid, i.e. completely determined as a Gorenstein toric contact 
manifold by its contact Betti numbers. More precisely, in each dimension, we show that its toric diagram, 
which we name small cross-polytope, is the unique toric diagram with its particular Ehrhart polynomial. 
We will also prove a rigidity result for a family of Gorenstein toric contact manifolds that arise as the primitive 
prequantization of monotone $\CP^1$-bundles over $\CP^{n-1}$.
\end{abstract}

\maketitle

\section{Introduction}
\label{sec:intro}

Gorenstein toric contact manifolds are good toric contact manifolds with zero first Chern class. 
They are in $1$-$1$ correspondence with toric diagrams, i.e. integral simplicial polytopes with
unimodular facets (see Section~\ref{sec:Gorenstein}):
\[
\Dd \subset \R^n  \longleftrightarrow (N^{2n+1}_\Dd, \xi_\Dd) \,.
\]
The Ehrhart polynomial of $\Dd$ counts its number of rational points and is defined by
\[
L_{\mathcal{D}}(t) := \#(t\mathcal{D} \cap \mathbb{Z}^n), \forall t \in \mathbb{Z}_{> 0}\,,
\]
where $t\mathcal{D}=\{t \bm{x} \in \mathbb{R}^n: \bm{x} \in \mathcal{D}\}$. The Ehrhart polynomial
$L_{\mathcal{D}}(t)$ is indeed a degree $n$ polynomial in $t$ and, as proved in~\cite{AMM}, its
coefficients are contact invariants of $(N^{2n+1}_\Dd, \xi_\Dd)$. More precisely, if we write the Ehrhart
polynomial in the form
\[
L_{\Dd} (t) = \sum_{k=0}^n h^\ast_k (\Dd) \binom{t+n-k}{n}\,,
\]
we have that the non-negative coefficients $h^\ast_k (\Dd) \in \Z^+_0$ are the contact invariants of 
$(N_{\Dd}, \xi_{\Dd})$ given by
\[
h^\ast_{n-k} (\Dd) = cb_{2k} (N_{\Dd}, \xi_{\Dd}) - cb_{2(k-1)} (N_{\Dd}, \xi_{\Dd})\,,
\]
where $cb_{\ast} (N_{\Dd}, \xi_{\Dd})$ are the \emph{contact Betti numbers}, i.e.
\[
cb_{\ast} (N_{\Dd}, \xi_{\Dd}) = \rank HC_\ast (N_\Dd, \xi_\Dd)
\] 
where  $HC_\ast (N_\Dd, \xi_\Dd)$ is the cylindrical contact homology with grading given by the SFT degree.
Note that for toric contact manifolds we have that $cb_\ast (N_\Dd, \xi_\Dd) = 0$ whenever $\ast <0$ or 
$\ast$ is odd, which means in particular that the Ehrhart polynomial determines and is determined by
the contact Betti numbers.
\begin{remark} \label{rem:well-defined}
As explained in~\cite{AMM}, despite some foundational technical difficulties in proving that cylindrical contact
homology is a well defined contact invariant, these contact Betti numbers are indeed well defined contact
invariants in many contexts. These include Gorenstein toric contact manifolds that have crepant 
(i.e. with zero first Chern class) toric symplectic fillings and all the particular examples considered in this 
paper have this property.
\end{remark}

In this paper we address the following natural question: 
\vskip 0.3 cm
\noindent \emph{To what extent do these contact invariants determine the Gorenstein toric contact manifold or, 
equivalently, to what extent does the Ehrhart polynomial determine a toric diagram $\Dd$ 
up to integral translation and $GL(n,\Z)$ transformations (i.e. up to unimodular equivalence)?} 
\vskip 0.3 cm
\begin{definition} \label{def:rigid}
A Gorenstein toric contact manifold is said to be \emph{rigid} if it is completely determined, as a toric contact 
manifold, by its contact Betti numbers. Equivalently, $(N_\Dd, \xi_\Dd)$ is rigid if its toric diagram $\Dd$ is 
\emph{Ehrhart determined}, i.e. whenever $\Dd'$ is a toric diagram with  $L_{\Dd'} (t) = L_{\Dd} (t)$ we have 
that $\Dd'$ is unimodularly equivalent to $\Dd$. A Gorenstein toric contact manifold is said to be \emph{flexible} 
if it is not rigid. Two toric diagrams $\Dd$ and $\Dd'$ are said to be \emph{Ehrhart equivalent} if they have the
same Ehrhart polynomial, i.e. if $L_{\Dd} (t) = L_{\Dd'} (t)$.
\end{definition}

The Ehrhart polynomial $L_{\Dd} (t)$ determines the volume of $\Dd\subset\R^n$ through its $t^n$ coefficient.
When $n=1$ the volume is the only invariant of a toric diagram $\Dd\subset\R$, up to unimodular equivalence,
and so we have that all Gorenstein toric contact $3$-manifolds are rigid. These are $S^3$ with its standard 
contact structure (toric diagram $=[0,1]$), and the Lens spaces $L^3_{p+1} (1,p)$, $p\in\N$, with standard 
quotient contact structure coming from $S^3$ (toric diagram $=[0,p+1]$).

Any odd-dimensional sphere $S^{2n+1}$ with its standard contact structure is an example of a rigid Gorenstein
toric contact manifold. In fact, its toric diagram is the unimodular simplex
\[
\Dd = \conv \{ \bm{0}, \bm{e}_1, \ldots, \bm{e}_n \} \subset\R^n\,,\ \text{where $\bm{e}_1, \ldots, \bm{e}_n$ are
the canonical basis vectors of $\R^n$,}
\]
which is well-known to be Ehrhart determined.

Despite these examples, flexibility is the norm and this can be seen already when $n=2$. For example,
given $p\in\N$, the toric contact manifolds 
\[
(S^3 \times S^2, \xi_p)\ \text {with toric diagram} \ \Dd_p = \conv\{(0,0), (1,0),(0,1), (p+1,p+1)\}
\]
and
\[
S^\ast (L^3_{p+1}(1,q)) \ \text {with toric diagram} \ \Dd_{p,q} = \conv\{(0,0), (1,0),(p+1,q), (p+1,q+1)\}\,,
\]
for all $q\in\N$ coprime with $p+1$, where $S^\ast (L^3_{p+1}(1,q))\cong L^3_{p+1}(1,q) \times S^2$ is
the unit cosphere bundle of $L^3_{p+1}(1,q)$, have the same contact Betti numbers
\[
cb_0 = p\,,\ cb_2 = 2p+1 \ \text{and}\ cb_{2k} = 2p+2\,,\ \forall\, k\geq n\,.
\]
See~\cite{AM0, AMM0} for more information on these examples.

In this paper we will describe some families of flexible and rigid Gorenstein toric contact 
manifolds that arise as the smooth prequantization of monotone toric symplectic manifolds. 
These families are interesting both from the toric contact geometry and the Ehrhart theory 
points of view. 

\begin{remark} \label{rem:monotone}
Recall that a symplectic manifold $(M^{2n}, \omega)$ is said to be \emph{monotone} when
the first Chern class $c_1 (M)$ is a multiple of $[\omega] \in H^2 (M; \R)$. In the toric
context of this paper, that multiple is always positive and, by default, we will assume that it is equal
to $1$, i.e. a monotone toric symplectic manifold will mean $(M^{2n}, \omega)$ with $[\omega] = c_1 (M)$. 
Moreover, by default, the prequantization of a monotone toric symplectic manifold $(M^{2n}, [\omega] = c_1 (M))$ 
will be the toric contact manifold $(N^{2n+1}, \xi)$ obtained as the $S^1$-bundle over $M^{2n}$ with Chern class 
$\frac{1}{2\pi}[\omega] = \frac{1}{2\pi} c_1 (M) \in H^2 (M; \Z)$, with $\xi$ being the horizontal distribution of a 
connection $1$-form with curvature $\frac{1}{2\pi}\omega$.
\end{remark}

As we recall in Section~\ref{sec:Gorenstein}, if $(M^{2n}, [\omega] = c_1 (M))$ is
a monotone toric symplectic manifold, its moment polytope $\Pp \subset \R^n$ can be
translated to be a reflexive Delzant polytope and its prequantization $(N^{2n+1}, \xi)$ 
is a Gorenstein toric contact manifold with toric diagram $\Dd=-\Pp^\ast \subset \R^n$, 
where $\Pp^\ast$ is the polar dual of $\Pp$:
\begin{align*}
\mathcal{P}^*:=\{ \bm{x} \in \mathbb{R}^n : \bm{x} \cdot \bm{y} \leq 1,  \forall \bm{y} \in \mathcal{P} \} \,.
\end{align*}
For example, when $M^{2n} = \CP^1 \times \cdots \times \CP^1$ with symplectic form $\omega$
such that $[\omega] = c_1(\CP^1 \times \cdots \times \CP^1)$, we have that $\Pp = [-1, 1]^n \subset \R^n$ 
is the (hyper-)cube and $\Dd = \diamond_n \subset \R^n$ is the cross-polytope:
\[
\diamond_n := \{(x_1, \ldots, x_n) \in \mathbb{R}^n : |x_1|+ \cdots +|x_n| \leq 1\} \quad\text{(see Figure~\ref{fig:0}).}
\]
\begin{figure}[]
	\centering
	\begin{tikzpicture}[scale=3, line join=bevel, tdplot_main_coords]
		\coordinate (e1) at (0,0,1);
		\coordinate (e2) at (1,0,0);
		\coordinate (e3) at (0,1,0);
		\coordinate (m1) at (0,0,-1);
		\coordinate (m2) at (-1,0,0);
		\coordinate (m3) at (0,-1,0);
	
		\draw (e2) -- (m2);
		
		\draw [fill opacity=0.2,fill=black, draw=none] (e2) -- (e3) -- (m2) -- (m3) -- cycle;

		\draw [dashed] (e1) -- (m2);
		\draw (e1) -- (m3);
		\draw (e2) -- (m1);
		\draw (e2) -- (m3);
		\draw (e3) -- (m1);
		\draw  [dashed] (e3) -- (m2);
		\draw (e1) -- (e3);
		\draw (e1) -- (e2);
		\draw (e2) -- (e3);
		\draw (m1) -- (m3);
		\draw  [dashed] (m1) -- (m2);
		\draw  [dashed] (m2) -- (m3);
		
	\end{tikzpicture}
	\caption{Cross-polytopes in dimensions $1$, $2$ and $3$.}
	\label{fig:0}
\end{figure}
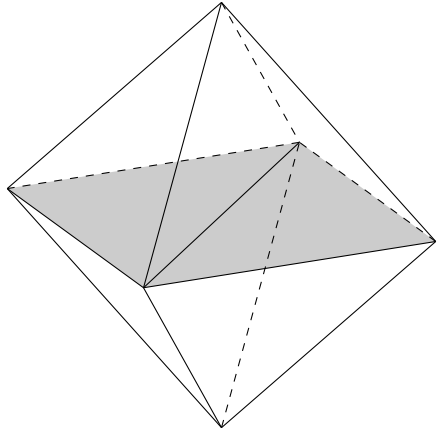
The cross-polytope is not Ehrhart determined and our first family arises from a particular class of
toric diagrams that are Ehrhart equivalent to the cross-polytope (see Section~\ref{sec:ppyr}). 
These turn out to correspond exactly to the Gorenstein toric contact manifolds that arise as prequantizations 
of monotone Bott manifolds $B_n$ (see Section~\ref{sec:FBm}). Bott manifolds are iterated $\CP^1 = S^2$-bundles 
of the form
\[
B_n \xrightarrow{\pi_{n}} B_{n-1} \xrightarrow{\pi_{n-1}} \cdots \xrightarrow{\pi_{2}} B_1 = \CP^1 \xrightarrow{\pi_{1}}
B_0 = \{\text{point}\}
\]
where
\[
B_i = \CP (\xi_i \oplus \C) \ \text{with}\ \xi_{i} \rightarrow B_{i-1}\ \text{a $\C$-line bundle.}
\]
\begin{proposition} \label{prop:FanoBott}
Let $(N^{2n+1}, \xi)$ be the prequantization of a monotone Bott manifold $(B^{2n}, [\omega] = c_1 (B^{2n}))$,  with
reflexive moment polytope $\Pp\subset\R^n$, and let $\Dd=-\Pp^\ast \subset \R^n$ be its toric diagram. Then
\[
h^\ast_k (\Dd) = \binom{n}{k}\,,\ k=0,\ldots,n\,.
\]
In particular, all these $(N^{2n+1}, \xi)$ are flexible and the corresponding toric diagrams provide a family of polytopes that are 
Ehrhart equivalent to the cross-polytope, i.e. for any such toric diagram $\Dd\subset\R^n$ we have that
\[
L_\Dd (t) = L_\diamond (t) = \sum_{k=0}^n \binom{n}{k} \binom{t+n-k}{n}\,.
\]
\end{proposition}
\noindent In Section~\ref{sec:ppyr} we will describe and enumerate, up to unimodular equivalence, all these 
toric diagrams. Note that when $n\geq 3$ these are not all the toric diagrams that are Ehrhart equivalent
to the cross-polytope. Hence, when $n\geq 3$ there are Gorenstein toric contact manifolds with the same 
contact Betti numbers of a monotone prequantizattion of a monotone Bott manifold that are not any of those 
manifolds. 

Our second family arises from the fact that the first Chern class of the trivial Bott manifold 
$\CP^1 \times \cdots \times \CP^1$ is divisible by 2 and we can consider an integral symplectic form 
$\omega$ with $[\omega] = c_1(\CP^1 \times \cdots \times \CP^1)/2$. The corresponding prequantization 
is still a Gorenstein toric contact manifold, whose toric diagram is what we call in this paper a small 
cross-polytope, see Figure~\ref{fig:1} and Section~\ref{sec:smallcrosspolytopes}. As we prove in Subsection~\ref{ssec:determined},
small cross-polytopes turn out to be Ehrhart determined and so we have the following theorem.
\begin{theorem} \label{thm:small}
 In any dimension $n\in\N$, the prequantization $(N^{2n+1}, \xi)$ of $\CP^1 \times \cdots \times \CP^1$
 with symplectic form $\omega$ such that $[\omega] = c_1(\CP^1 \times \cdots \times \CP^1)/2$ is a rigid 
 Gorenstein toric contact manifold.
\end{theorem}
\begin{remark}
Small cross-polytopes share some interesting combinatorial properties with cross-polytopes. See 
Section~\ref{sec:smallcrosspolytopes}, in particular Theorem~\ref {teo:4}, which implies that the 
contact manifold $(N^{2n+1}, \xi)$ of Theorem~\ref{thm:small} has a crepant toric symplectic filling, 
and its final Subsection~\ref{ssec:remarks}.

\end{remark}
\begin{figure}[]
	\centering
	\begin{tikzpicture}[scale=3, line join=bevel, tdplot_main_coords]
		\coordinate (e1) at (0,0,0);
		\coordinate (e2) at (1,0,0);
		\coordinate (e3) at (0,1,0);
		\coordinate (m1) at (0,0,1);
		\coordinate (m2) at (-1,0,1);
		\coordinate (m3) at (0,-1,1);

		\draw [fill opacity=0.2,fill=black,draw=none] (e1) -- (e2) -- (e3) -- cycle;
		\draw [fill opacity=0.2,fill=black] (m1) -- (m2) -- (m3) -- cycle;		
		\draw [dashed] (e1) -- (m2);
		\draw [dashed] (e1) -- (m3);
		\draw (e2) -- (m1);
		\draw (e2) -- (m3);
		\draw (e3) -- (m1);
		\draw (e3) -- (m2);
		\draw [dashed] (e1) -- (e3);
		\draw [dashed] (e1) -- (e2);
		\draw (e2) -- (e3);
		\draw (m1) -- (m3);
		\draw (m1) -- (m2);
		\draw (m2) -- (m3);
		
	\end{tikzpicture}
	\caption{Small cross-polytope in dimension $3$.}
	\label{fig:1}
\end{figure}
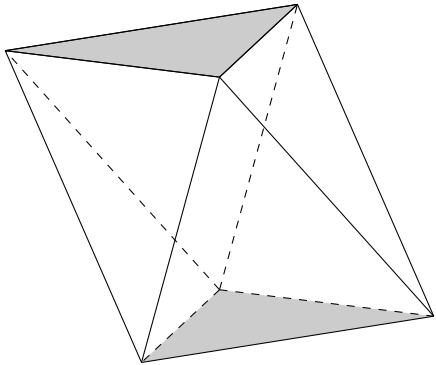

Our third family arises again from primitive prequantizations of monotone toric symplectic manifolds with
non-primitive first Chern class. The examples in this family are not rigid, but the family itself is rigid in the
sense that it contains all Gorenstein toric contact manifolds with these particular contact invariants. For 
any $n\geq 2$, the monotone toric symplectic manifolds that give rise to this family are
\[
M^{2n}_k = \CP (\xi_k \oplus \C) \ \text{where}\ \xi_k \to \CP^{n-1} \ \text{is a $\C$-line bundle
with $c_1 (\xi_k) = k \in \Z$.}
\]
When $0\leq k < n$ we have that $M^{2n}_k$ is monotone, i.e. it has a toric symplectic form $\omega_k$ with
$[\omega_k] = c_1 (M_k^{2n})$. When in addition $k \equiv n \pmod 2$ we have 
that $c_1 (M_k^{2n})$ is divisible by $2$. The Gorenstein toric contact manifolds $(N^{2n+1}_k, \xi)$ 
in this family are the prequantizations of these $(M^{2n}_k, \omega)$ with 
$[\omega] = c_1 (M_k^{2n}) /2$ and $0\leq k < n$, $k \equiv n \pmod 2$. As we will show in 
Section~\ref{sec:bipysimplex}, the Ehrhart polynomial of the corresponding toric diagrams $\Dd_k \subset \R^n$ 
is independent of $k$ and given by
\[
L_{\Dd_k} (t) = \sum_{j=0}^{n-1} \binom{t+n-j}{n}\,,\ \text{i.e. $h^\ast_0 = \cdots = h^\ast_{n-1}=1$ and $h^\ast_n = 0$.}
\]
Moreover, we have the following rigidity theorem (cf. Theorem~\ref{thm:bipysimplex2}).
\begin{theorem} \label{thm:bipysimplex1}
Let $(N^{2n+1}_\Ss, \xi_\Ss)$ be a Gorenstein toric contact manifold determined by a toric diagram $\Ss\subset\R^n$
such that
\[
L_{\Ss} (t) = \sum_{j=0}^{n-1} \binom{t+n-j}{n}\,.
\]
Then $\Ss$ is unimodularly equivalent to $\Dd_k$ for a unique $k\in\Z$ satisfying $0\leq k < n$, $k \equiv n \pmod 2$,
and $(N^{2n+1}_\Ss, \xi_\Ss)$ is isomorphic as a toric contact manifold to the corresponding $(N^{2n+1}_k, \xi)$.
\end{theorem}

\begin{remark} \label{rem:rem}
All the rigidity examples considered in Theorems~\ref{thm:small} and~\ref{thm:bipysimplex1} arise from primitive
prequantizations of monotone toric symplectic manifolds with divisible (non-primitive) first Chern class. In other 
words, their minimal Chern number is bigger than $1$, in fact equal to $2$ in all these examples. It would be 
interesting to find rigid examples without this feature but we do not know if they exist.
\end{remark}

\vskip .3cm
\noindent {\bf Organization of the paper. } 
Section~\ref{sec:tsm} contains a very brief description of toric symplectic manifolds and their Delzant moment polytopes. 
Section~\ref{sec:FBm} introduces monotone Bott manifolds and reviews some recent results on their classification. 
Section~\ref{sec:Gorenstein} contains a brief description of Gorenstein toric contact manifolds and their toric diagrams. 
Contact Betti numbers and their values for the main examples in this paper are the subject of Section~\ref{sec:invariants}. 
The second part of the paper is devoted to Ehrhart theory of toric diagrams. 
Section~\ref{sec:polyEhrhart} contains some background, main examples and basic combinatorial results. 
Section~\ref{sec:smallcrosspolytopes} is devoted to small cross-polytopes, their Ehrhart theory and the proof of
Theorem~\ref{thm:small}.
Section~\ref{sec:monotone} proves a useful result in the Ehrhart theory of toric diagrams that arise from
prequantizations of divisible reflexive Delzant polytopes and Remark~\ref{rem:main} shows how it can be
translated as a known result regarding their cylindrical contact homology.
Section~\ref{sec:ppyr} is devoted to the double-pyramid construction of toric diagrams Ehrhart equivalent to
cross-polytopes and their enumeration up to unimodular equivalence. Finally, Section~\ref{sec:bipysimplex}
introduces the toric diagrams of primitive prequantizations of monotone $\CP^1$-bundles over $\CP^{n-1}$ and 
proves Theorem~\ref{thm:bipysimplex1}.
 
\vskip .3cm
\noindent {\bf Acknowledgements.}
The first and third author are grateful to Sinai Robins for pointing out reference~\cite{R-V}. 
The third author acknowledges the tool Gemini for useful assistance.

\section{Toric symplectic manifolds}
\label{sec:tsm}

A \emph{toric symplectic manifold} is a closed symplectic manifold $(M, \omega)$ of dimension $2n$ 
equipped with an effective hamiltomian $n$-torus action, i.e.
\[
\tau: \T^n \cong \R^{n} / 2\pi \Z^n \hookrightarrow \Ham (M, \omega)\,.
\]
The $n$ hamiltonian functions $(h_1, \ldots, h_n)$ generating these $n$ commuting circle actions
give rise to the \emph{moment map}
\[
\bm{\mu} : M \to \R^n\,,\ \bm{\mu} = (h_1, \ldots, h_n)\,.
\]
The Atiyah-Guillemin-Sternberg convexity theorem implies that the moment map image, $\Pp:= \bm{\mu} (M)$, 
is the convex hull of the images of the fixed points  of the action, i.e. $\Pp = \conv (\bm{\mu}(M^{\T^n}))$, and
can be described as a finite intersection of half-spaces:
\[
\Pp = \cap_{j=1}^d \left\{ \bm{x}\in\R^n\,:\ \bm{x}\cdot\bm{\nu}_j + \lambda_j \geq 0 \right\}\,,
\]
where $\bm{\nu}_j \in\Z^n$ is the primitive integral interior normal to the $j$-th facet of $P$ and
$\lambda_j \in \R$ determines its position relative to the origin. 
$\Pp\subset\R^n$ is called the \emph{moment polytope} of the toric symplectic manifold $(M, \omega, \tau)$.
It is a Delzant polytope and a complete invariant of $(M, \omega, \tau)$ \cite{D}. 
Note that the hamiltonian functions $(h_1, \ldots, h_n)$ are only determined up 
to constants and hence the moment polytope $\Pp$ is only determined up to translation in $\R^n$.

The real numbers $\lambda_j$, $j=1,\ldots, d$, can be used to determine the cohomology class of the 
symplectic form $\omega$ and we have that $[\omega] = c_1 (M)$, i.e.  $(M, \omega, \tau)$ is
a monotone toric symplectic manifold, if and only if there is a translation of $\Pp$ such that all the 
$\lambda_j$'s are equal to $1$, i.e.
\[
\Pp = \cap_{j=1}^d \left\{ \bm{x}\in\R^n\,:\ \bm{x}\cdot\bm{\nu}_j + 1 \geq 0 \right\}\,.
\]

\begin{example} \label{ex:2-sphere}
When $n=1$ there is only one toric symplectic manifold: the $2$-sphere $S^2 = \CP^1$ equipped with an area form
and circle action given by rotation around an axis. For the monotone area form, with total area $4\pi$, we
have that 
\[
\Pp = [-1,1] = \left\{x\in\R\,:\ -x+1\geq 0\right\} \cap  \left\{x\in\R\,:\ x+1\geq 0\right\}\,.
\]
\end{example}

\begin{example} \label{ex:hirzebruch}
Hirzebruch surfaces are examples of toric symplectic $4$-manifolds ($n=2$) that can be constructed as $\CP^1$-bundles
over $\CP^1$. More precisely, the Hirzebuch surface $H_k$, $k\in\Z$, is given by
\[
H_k := \CP (\Oo (k) \oplus \C) \to \CP^1 = S^2\,.
\]
The primitive, integral, interior normals to the facets of its moment polytope can be listed as the columns of the matrix
\[
\begin{bmatrix}
1 & 0 & -1 & 0  \\
0 & 1 & k & -1 \\
\end{bmatrix}\,.
\]
See Figure~\ref{fig1} with a moment polytope for $H_2$.
\begin{figure}[ht]
\includegraphics[width=4in, height=2.4in]{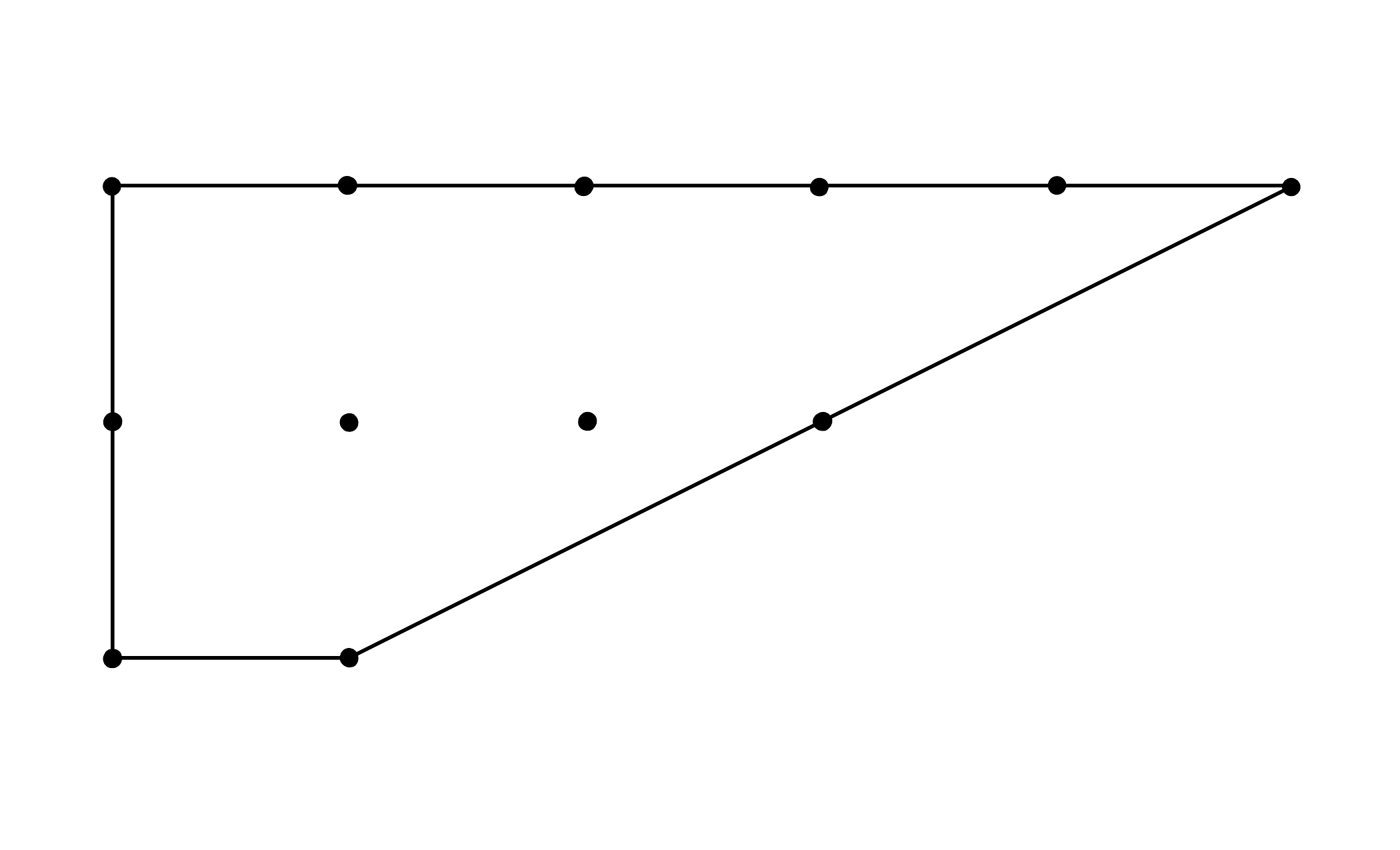}
\caption{A moment polytope for $H_2$.}
\label{fig1}
\end{figure}

There are only two diffeomorphism types of Hirzebruch surfaces: all even $H_{2k}$ are diffeomorphic to $S^2 \times S^2$,
while all odd $H_{2k-1}$ are diffeomorphic to $S^2 \tilde{\times} S^2 =$ the nontrivial $S^2$-bundle over $S^2$.

There are only two monotone Hirzebruch surfaces: $H_0$ and $H_1 \cong H_{-1}$. In fact, the complete list of monotone
toric symplectic $4$-manifolds is
\[
H_0 \cong \CP^1 \times \CP^1\,,\ \CP^2\,,\ H_1 \cong \CP^2 \# \oCP^2 \cong S^2 \tilde{\times} S^2\,,\  \CP^2 \# 2\oCP^2\ \text{and}\ 
 \CP^2 \# 3\oCP^2\,.
\]
See Figure~\ref{fig2} with their moment polytopes.
\begin{figure}[ht]
\includegraphics[width=4in, height=2.4in]{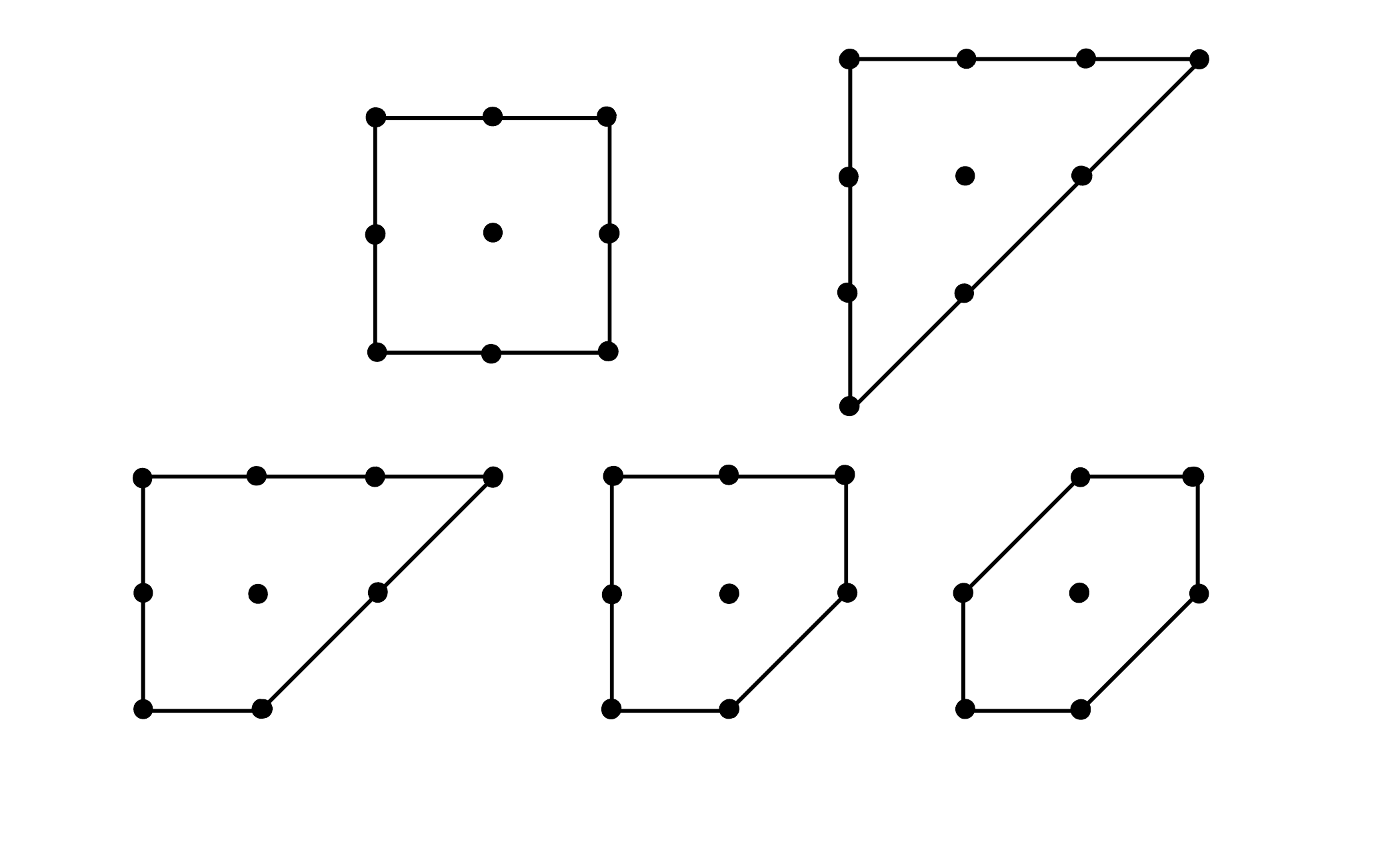}
\caption{Moment polytopes for the monotone toric symplectic $4$-manifolds: 
$H_0 \cong \CP^1 \times \CP^1$ and $\CP^2$ on the top row, 
$H_1 \cong \CP^2 \# \oCP^2 \cong S^2 \tilde{\times} S^2$, $\CP^2 \# 2\oCP^2$ and $\CP^2 \# 3\oCP^2$ on the bottom row.}
\label{fig2}
\end{figure}
\end{example}

\section{Monotone Bott manifolds as toric symplectic manifolds}
\label{sec:FBm}

Bott manifolds, introduced in~\cite{GK}, iterate the construction of Hirzebruch surfaces.
More precisely, they are iterated $\CP^1 = S^2$-bundles of the form
\[
B_n \xrightarrow{\pi_{n}} B_{n-1} \xrightarrow{\pi_{n-1}} \cdots \xrightarrow{\pi_{2}} B_1 = \CP^1 \xrightarrow{\pi_{1}}
B_0 = \{\text{point}\}
\]
where
\[
B_i = \CP (\xi_i \oplus \C) \ \text{with}\ \xi_{i} \rightarrow B_{i-1}\ \text{a $\C$-line bundle.}
\]
At each stage, we always have that $H^\ast (B_i ; \Z) \cong H^\ast ((S^2)^i ; \Z)$ as graded abelian groups (but not as rings),
hence $\rank H^{2k} (B_i;\Z) = \binom{i}{k}$. In particular, $\rank H^{2} (B_i;\Z) = i$ and so the Chern class $c_1(\xi_i)$ of
each $\C$-line bundle $\xi_{i} \rightarrow B_{i-1}$ is determined by $i-1$ integers $(a_{i,1}, \ldots, a_{i,i-1})$.

All Bott manifolds are toric symplectic manifolds determined by moment polytopes that are combinatorially equivalent to
hypercubes. For a Bott manifold $B_n$ of dimension $2n$, the $2n$ primitive, integral, interior normals of the corresponding
moment polytope are determined by the Chern classes $c_1(\xi_i)$, $i=2, \ldots, n$. In an appropriate integral basis of 
$\Z^n \subset \R^n$, these normals are given by the $n$ coordinate basis vectors, i.e. the colums of the $(n\times n)$ 
identity matrix $I_n$, and the $n$ ``opposite'' vectors given by the colums of the lower triangular matrix
\[
L(B_n) = 
\begin{bmatrix}
-1 & \  & \  & \  & \  \\
a_{2,1} & -1  & \  & \  & \  \\
a_{3,1}  & a_{3,2}  & -1  & \  & \  \\
\vdots  & \vdots   & \ddots   & \ddots  & \  \\ 
a_{n,1}  & a_{n,2}  & \cdots  & a_{n,n-1}  & -1 \\
\end{bmatrix} \,.
\]

Necessary and sufficient conditions on $L(B_n)$ so that $B_n$ is Fano, i.e. so that $c_1 (B_n) > 0$, were given in~\cite{S}. 
From the point of view of toric symplectic geometry, these are precisely the monotone Bott manifolds. 
\begin{theorem}[\!\!\cite{S}] \label{thm:S}
$B_n$ is a monotone toric symplectic manifod if and only if $L(B_n)$ is such that for any colum $j\in\{1,\ldots,n-1\}$ 
one of the folowing holds:
\begin{itemize}
\item[(1)] $a_{j+1,j} = \cdots = a_{n,j} = 0$;
\item[(2)] $\exists\ q\in\{j+1, \ldots,n\}\,:\ a_{q,j} =1$ and $a_{i,j} = 0$ for all $i\in\{j+1, \ldots,n\}\setminus\{q\}$;
\item[(3)] $\exists\ q\in\{j+1, \ldots,n\}\,:\ a_{q,j} = -1$, $a_{i,j} = 0$ for all $i\in\{j+1,\ldots, q-1\}$ and $a_{i,j} = a_{i,q}$
for all $i\in\{q+1,\ldots, n\}$.
\end{itemize}
\end{theorem}
\noindent A toric cohomological rigidity result for Fano Bott manifolds, i.e. for monotone Bott manifolds, was proved 
in~\cite{CLMP-1}.
\begin{theorem}[\!\!\cite{CLMP-1}]
Given two monotone Bott manifolds $B$ and $B'$, the following are equivalent:
\begin{itemize}
\item[(1)] $B\cong B'$ as toric manifolds (i.e. equivariantly diffeomorphic);
\item[(2)] $H^\ast (B;\Z) \cong_{c_1} H^\ast (B';\Z)$, i.e. $c_1$-preserving graded ring isomorphism;
\item[(3)] $L(B)$ can be transformed into $L(B')$ using certain operations.
\end{itemize}
\end{theorem}
\noindent Moreover, in~\cite{CLMP-2} using a result of~\cite{HK}, an algorithm is provided to 
enumerate all toric isomorphism classes of monotone Bott manifolds of any given $\C$-dimension $n$ 
and implemented it up to $n=10$. The result is shown in the following table, which also includes 
the number of monotone toric symplectic manifolds following their classification in~\cite{O}.
\[
\begin{array}{|c|c|c|c|c|c|c|c|c|c|c|}
\hline
\dim_\C & 1 & 2 & 3 & 4 & 5 & 6 & 7 & 8 & 9 & 10 \\
\hline
\# \ \text{monotone Bott} & 1 & 2 & 5 & 13 & 37 & 111 & 345 & 1105 & 3624 & 12099 \\
\hline
\# \ \text{monotone toric} & 1 & 5 & 18 & 124 & 866 & 7622 & 72256 & 749892 & 8229721 & ? \\
\hline
\end{array}
\]

\begin{remark}
In Section~\ref{sec:ppyr} we present an alternative approach to this enumeration of all
toric isomorphism classes of monotone Bott manifolds.
\end{remark}

\begin{example}
The five monotone Bott manifolds in $\C$-dimension $3$ can be listed in the following way:
\[
\text{1)} \ \CP^1 \times \CP^1 \times \CP^1\,,\ 
L = 
\begin{bmatrix}
-1 & 0 & 0  \\
0 & -1 & 0 \\
0 & 0 & -1 \\
\end{bmatrix}
\qquad 
\text{2)} \ \CP^1 \times H_1 \,,\ 
L = 
\begin{bmatrix}
-1 & 0 & 0  \\
0 & -1 & 0 \\
0 & 1 & -1 \\
\end{bmatrix}
\]
\[
\text{3)} \ \CP (\Oo(1,-1)\oplus \C) \to \CP^1 \times \CP^1\,,\ 
L = 
\begin{bmatrix}
-1 & 0 & 0  \\
0 & -1 & 0 \\
1 & -1 & -1 \\
\end{bmatrix}
\]
\[
\text{4)} \ \CP (\Oo(1,1)\oplus \C) \to \CP^1 \times \CP^1\,,\ 
L = 
\begin{bmatrix}
-1 & 0 & 0  \\
0 & -1 & 0 \\
1 & 1 & -1 \\
\end{bmatrix}
\]
\[
\text{5)} \ \CP (\Oo(0,1)\oplus \C) \to H_1\,,\ 
L = 
\begin{bmatrix}
-1 & 0 & 0  \\
1 & -1 & 0 \\
0 & 1 & -1 \\
\end{bmatrix}
\]

See Figure~\ref{fig3} with their moment polytopes.
\begin{figure}[ht]
\includegraphics[width=4in, height=2.4in]{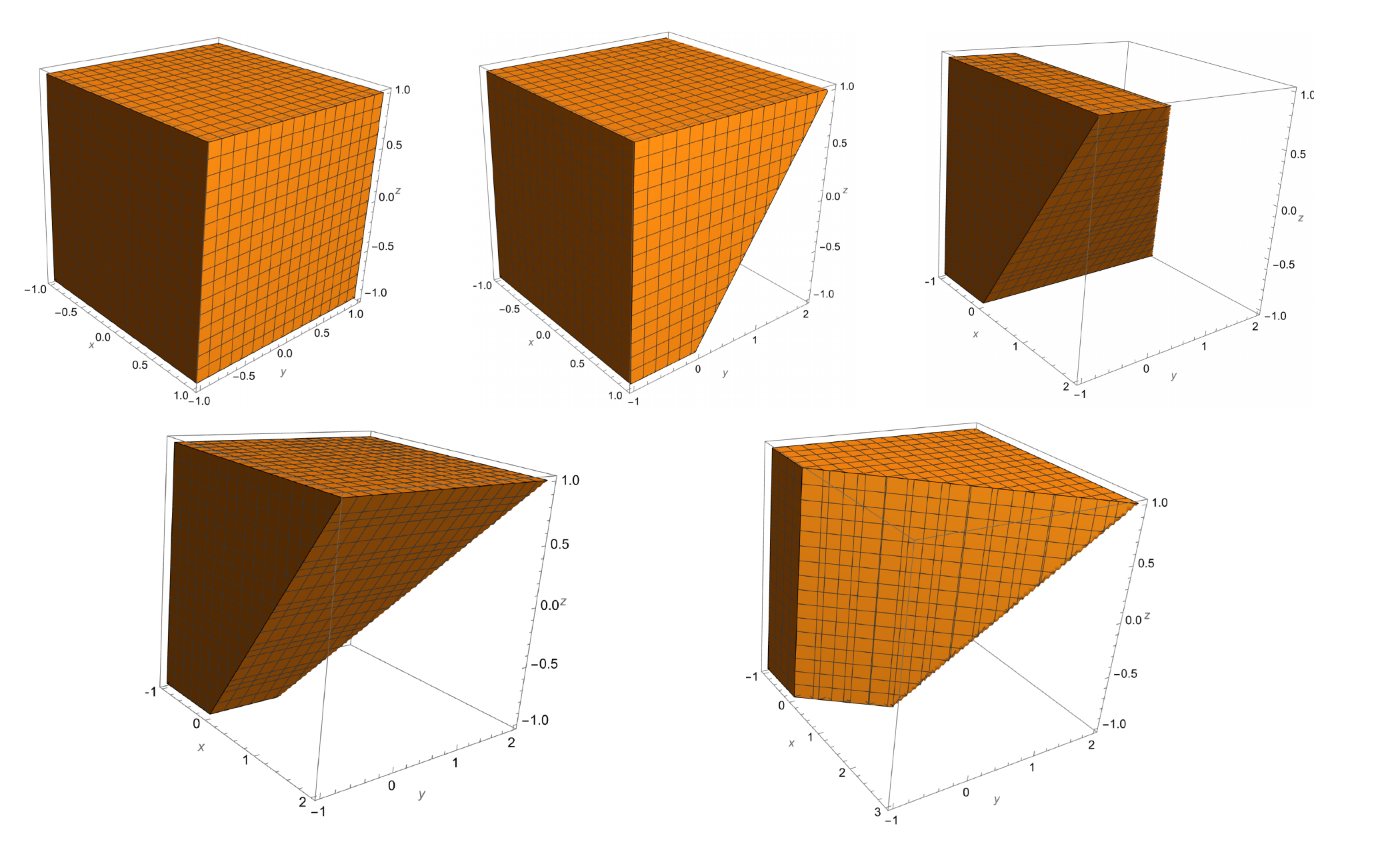}
\caption{Moment polytopes for the five monotone Bott $6$-manifolds: 
$\CP^1 \times \CP^1 \times \CP^1$, $\CP^1 \times H_1$ and $\CP (\Oo(1,-1)\oplus\C) \to \CP^1 \times \CP^1$ on the top row, 
$\CP (\Oo(1,1)\oplus\C) \to \CP^1 \times \CP^1$ and $\CP (\Oo(0,1)\oplus\C) \to H_1$ on the bottom row.}
\label{fig3}
\end{figure}
\end{example}

\section{Gorenstein toric contact manifolds}
\label{sec:Gorenstein}

In this section we will briefly recall the $1$-$1$ correspondence between Gorenstein toric 
contact manifolds, i.e. good toric contact manifolds (in the sense of~\cite{L}) with zero first 
Chern class, and toric diagrams (defined below). We follow closely the presentations in~\cite{AM, AMM}, 
which contain further details.

Via symplectization, there is a $1$-$1$ correspondence between co-oriented contact manifolds
and symplectic cones, i.e. triples $(W,\om,X)$ where $(W,\om)$ is a connected symplectic manifold
and $X$ is a vector field, the Liouville vector field, generating a proper $\R$-action
$\rho_t:W\to W$, $t\in\R$, such that $\rho_t^\ast (\om) = e^{t} \om$. A closed symplectic cone is a 
symplectic cone $(W,\om,X)$ for which the corresponding contact manifold $N = W/\R$ is closed.

A toric contact manifold is a closed contact manifold of dimension $2n+1$ equipped with an effective 
hamiltonian action of the standard torus of dimension $n+1$: $\T^{n+1} = \R^{n+1} / 2\pi\Z^{n+1}$. 
Also via symplectization, toric contact manifolds are in $1$-$1$ correspondence with toric symplectic cones, 
i.e. closed symplectic cones $(W,\om,X)$ of dimension $2(n+1)$ equipped with an effective $X$-preserving 
hamiltonian $\T^{n+1}$-action, with moment map $\bm{\mu} : W \to \R^{n+1}$ such that $\bm{\mu} (\rho_t (w)) = e^{t} \bm{\mu} (w)$, 
for all $w\in W$ and $t\in\R$. Its moment cone is defined to be $C:= \bm{\mu}(W) \cup \{ 0\} \subset \R^{n+1}$.

A toric contact manifold is {\it good} if its toric symplectic cone has a moment cone with the following properties.
\begin{definition} \label{def:good}
A cone $\Cc\subset\R^{n+1}$ is \emph{good} if it is strictly convex and there exists a minimal set 
of primitive vectors $\bm{\nu}_1, \ldots, \bm{\nu}_d \in \Z^{n+1}$, with 
$d\geq n+1$, such that
\begin{itemize}
\item[(i)] $\Cc = \bigcap_{j=1}^d \{\bm{x}\in\R^{n+1}\mid \bm{x}\cdot\bm{\nu}_j \geq 0\}$.
\item[(ii)] Any codimension-$k$ face of $\Cc$, $1\leq k\leq n$, is the intersection of exactly $k$ facets whose set of 
normals can be completed to an integral basis of $\Z^{n+1}$.
\end{itemize}
The primitive vectors $\bm{\nu}_1, \ldots, \bm{\nu}_d \in \Z^{n+1}$ are called the defining normals of the good cone $\Cc\subset\R^{n+1}$.
\end{definition}
The analogue for good toric contact manifolds of Delzant's classification theorem for closed toric
symplectic manifolds is the following result (see~\cite{L}).
\begin{theorem} \label{thm:good}
For each good cone $\Cc\subset\R^{n+1}$ there exists a unique closed toric symplectic cone
$(W_{\Cc}, \om_{\Cc}, X_{\Cc}, \bm{\mu}_{\Cc})$ with moment cone $\Cc$.
\end{theorem}

The Chern classes of a co-oriented contact manifold can be canonically identified with the 
Chern classes of the tangent bundle of the associated symplectic cone. The following proposition 
gives a moment cone characterization for zero first Chern class, i.e. for the toric contact manifold 
to be Gorenstein.
\begin{prop} \label{prop:c_1}
Let $(W_{\Cc}, \om_{\Cc}, X_{\Cc}, \bm{\mu}_{\Cc})$ be a good toric symplectic cone with $c_1 (TW_C) = 0$. 
Then there exists an integral basis of $\T^{n+1}$ for which the defining normals $\bm{\nu}_1,\ldots,\bm{\nu}_d \in \Z^{n+1}$ 
of the corresponding moment cone $\Cc\subset\R^{n+1}$ are of the form
\[
\bm{\nu}_j = (\bm{v}_j, 1)\,,\ \bm{v}_j \in \Z^n\,,\ j=1,\ldots,d\,.
\]
\end{prop}
\noindent Under this condition, we will encode the Gorenstein toric contact manifold  whose symplectization is 
$(W_{\Cc}, \om_{\Cc}, X_{\Cc}, \bm{\mu}_{\Cc})$ by the polytope 
\[
\Dd = \conv(\bm{v}_1, \ldots, \bm{v}_d)\subseteq \R^n\,.
\]
\begin{proposition} \label{prop:diagram}
$\Dd$ is a toric diagram (cf. Definition~\ref{def:diagram}).
\end{proposition}
\begin{rem} \label{rem:diagram}
The conditions imposed on toric diagrams $\Dd=\conv(v_1, \ldots, v_d)$ are equivalent to the corresponding cone 
with normals $\bm{\nu}_j=(\bm{v}_j, 1)$, $j=1,\ldots,d$, being good.
\end{rem}
\begin{thm} \label{thm:diagram}
For each toric diagram $\Dd\subset\R^n$ there exists a unique Gorenstein toric contact manifold $(N_{\Dd}, \xi_{\Dd})$ of 
dimension $2n+1$.
\end{thm}

The main examples in the context of this paper are the so-called prequantizations of monotone toric symplectic 
manifolds. 
\begin{prop} \label{prop: prequant}
Let $(M^{2n}, \omega)\,,\ [\omega]=c_1(M)$, be a monotone toric symplectic manifold with moment polytope
\[
\Pp = \cap_{j=1}^d \left\{ \bm{x}\in\R^n\,:\ \bm{x}\cdot\bm{\nu}_j + 1 \geq 0 \right\}\,.
\]
Then
\[
\Dd = \conv (\bm{\nu}_1, \ldots, \bm{\nu}_d)  \subset \R^n
\]
is a toric diagram and the Gorenstein toric contact manifold $(N^{2n+1}_{\Dd}, \xi_{\Dd})$ is the prequantization of
$(M^{2n}, \omega)$.
\end{prop}
\begin{remark}
In fact, $\Pp \subset\R^n$ is a reflexive polytope and $\Dd = -\Pp^\ast \subset \R^n$, where $\Pp^\ast$ is its polar dual 
polytope.
\end{remark}
\begin{example}
Figure~\ref{fig4}  has the toric diagrams of the prequantization of the two monotone Bott manifolds when $n=2$:
$H_0 \cong \CP^1 \times \CP^1$ and $H_1 \cong \CP^2 \# \oCP^2$. Figure~\ref{fig5}  has the toric diagrams of the prequantization of the five monotone Bott manifolds when $n=3$:
$\CP^1 \times \CP^1 \times \CP^1$, $\CP^1 \times H_1$, $\CP (\Oo(1,-1)\oplus\C) \to \CP^1 \times \CP^1$, 
$\CP (\Oo(1,1)\oplus\C) \to \CP^1 \times \CP^1$ and $\CP (\Oo(0,1)\oplus\C) \to H_1$.
\begin{figure}[ht]
\includegraphics[width=2in, height=1.2in]{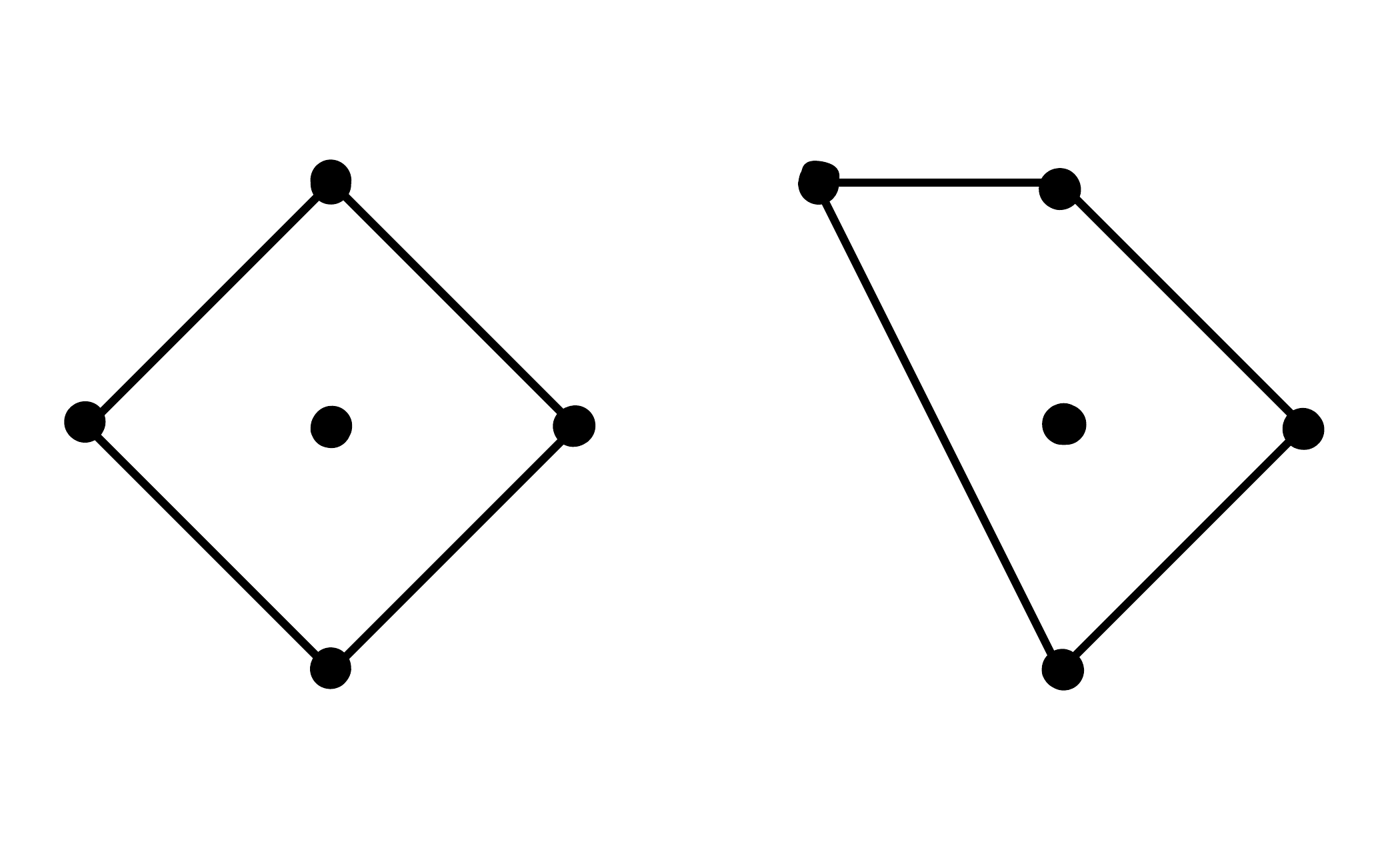}
\caption{Toric diagrams of the prequantization of the two monotone Bott manifolds when $n=2$:
$H_0 \cong \CP^1 \times \CP^1$ and $H_1 \cong \CP^2 \# \oCP^2$.}
\label{fig4}
\end{figure}
\begin{figure}[ht]
\includegraphics[width=4in, height=2.4in]{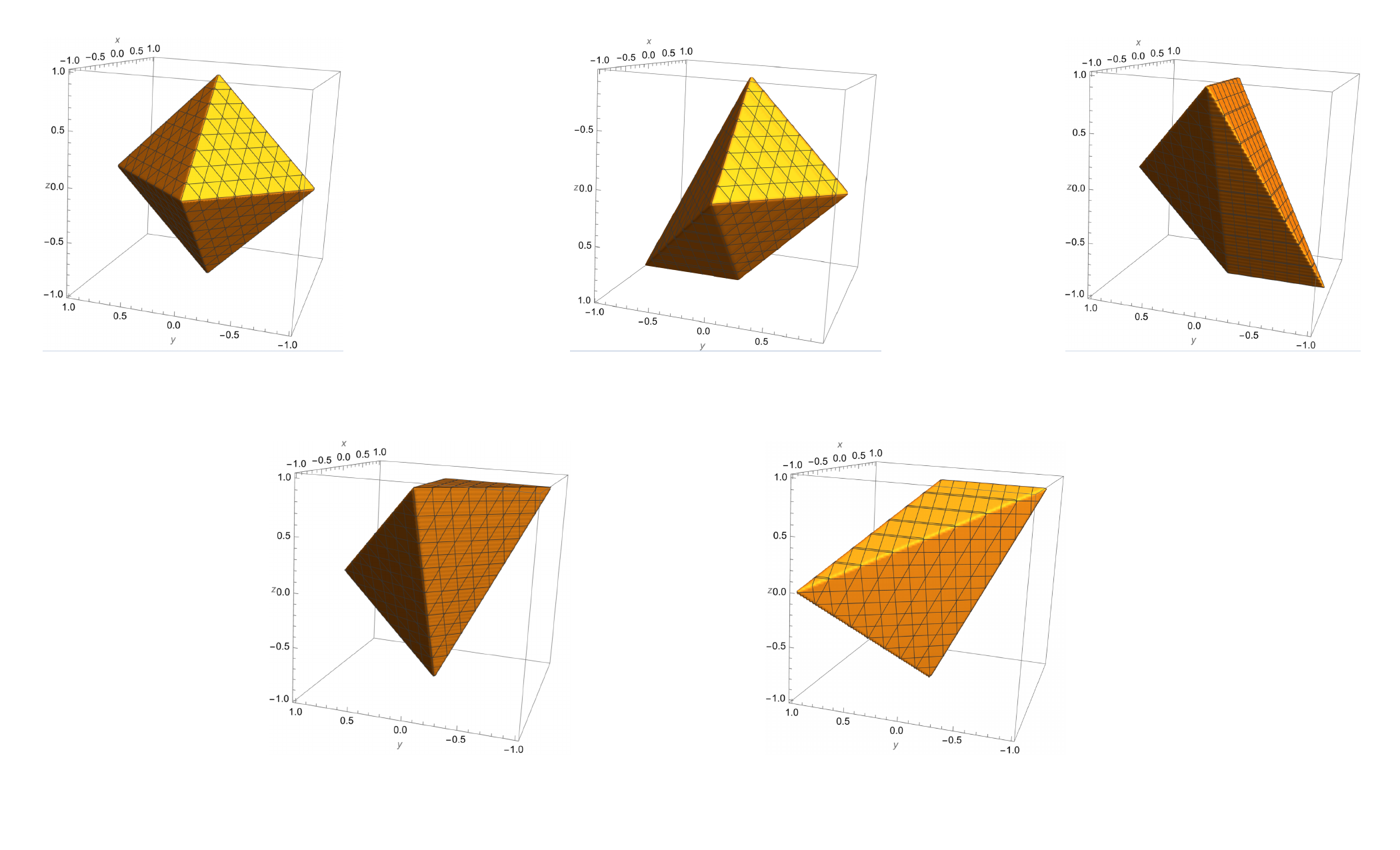}
\caption{Toric diagrams of the prequantization of the five monotone Bott manifolds when $n=3$:
$\CP^1 \times \CP^1 \times \CP^1$, $\CP^1 \times H_1$ and $\CP (\Oo(1,-1)\oplus\C) \to \CP^1 \times \CP^1$ on the top row, 
$\CP (\Oo(1,1)\oplus\C) \to \CP^1 \times \CP^1$ and $\CP (\Oo(0,1)\oplus\C) \to H_1$ on the bottom row.}
\label{fig5}
\end{figure}
\end{example}

\section{Contact invariants}
\label{sec:invariants}

As already described in the introduction, the coefficients of the Ehrhart polynomial of a toric diagram $\Dd\subset\R^n$
are contact invariants of the corresponding toric contact manifold $(N_{\Dd}, \xi_{\Dd})$. In particular, if we write the Ehrhart
polynomial in the form
\[
L_{\Dd} (t) = \sum_{k=0}^n h^\ast_k (\Dd) \binom{t+n-k}{n}\,,
\]
we have that the non-negative coefficients $h^\ast_k (\Dd) \in \Z^+_0$ are the contact invariants of $(N_{\Dd}, \xi_{\Dd})$ 
given by
\[
h^\ast_{n-k} (\Dd) = cb_{2k} (N_{\Dd}, \xi_{\Dd}) - cb_{2(k-1)} (N_{\Dd}, \xi_{\Dd})\,,
\]
where $cb_{\ast} (N_{\Dd}, \xi_{\Dd})$ are the \emph{contact Betti numbers}, i.e.
\[
cb_{\ast} (N_{\Dd}, \xi_{\Dd}) = \rank HC_\ast (N_\Dd, \xi_\Dd) \ \text{where  $HC_\ast (N_\Dd, \xi_\Dd)$ 
is the cylindrical contact homology.}
\] 
Recall that $HC_\ast (N_\Dd, \xi_\Dd) = 0$ whenever $\ast <0$ or $\ast$ is odd.

For the prequantization of monotone toric symplectic manifolds $(M^{2n}, \omega)\,,\ [\omega]=c_1(M)$, with reflexive 
moment polytope $\Pp \subset \R^n$ and toric diagram $\Dd= -\Pp^\ast \subset \R^n$, a theorem in~\cite{B} implies 
that 
\begin{equation} \label{eq:Betti}
h^\ast_k (\Dd) = \dim H^{2k} (M;\Q) \,.
\end{equation}
See~\cite{AMM} for a direct toric proof and recall that $H^{\text{odd}} (M; \Q) = 0$ for any toric symplectic manifold $M$.

\subsection{Contact invariants of the prequantization of monotone Bott manifolds}
\label{ssec:preqFanoBott}

For the prequantization of monotone Bott manifolds, we have that $\dim H^{2k} (B^{2n}; \Q) = \binom{n}{k}$,
$k=0,\ldots,n$, and so~(\ref{eq:Betti}) implies Proposition~\ref{prop:FanoBott}, i.e. the toric diagram $\Dd\subset\R^n$
of the Gorenstein prequantization $(N^{2n+1}, \xi)$ of any monotone Bott manifold $(B^{2n}, [\omega] = c_1)$ is such that
\[
h^\ast_k (\Dd) = \binom{n}{k}\,,\ k=0,\ldots,n\,.
\]
This implies that the non-zero contact Betti numbers of any such $(N^{2n+1}, \xi)$ are given by
\[
cb_{2k} (N^{2n+1}, \xi) = \sum_{j=0}^k  \binom{n}{j}\,,\ k=0, \ldots, n\,, \ \text{and} \ cb_{2k} (N^{2n+1}, \xi) = 2^n\,,\ \forall k>n\,.
\]
\begin{equation*}
\begin{array}{|c|c|c|c|c|c|c|cl} \hline
  \ast =  & \quad   0 \quad  & \quad 2 \quad & \quad 4 \quad & \quad \cdots \quad & \quad 2(n-1) \quad & \quad 2n \quad & \text{even}\, > 2n
\\ 
\hline 
cb_\ast (N^{2n+1},\xi) &   1   & 1+n & \frac{n(n+1)}{2} +1 & \cdots & 2^n -1 & 2^n &  2^n
\\
\hline
\end{array}
\end{equation*}
\begin{remark} 
It follows from Theorem~\ref{thm:S} that these toric diagrams are precisely the ones described and enumerated in 
Section~\ref{sec:ppyr}.
\end{remark}

\subsection{Contact invariants of the primitive prequantization of $\CP^1 \times \dots \times \CP^1$}
\label{ssec:primitive1}

The first Chern class of the trivial Bott manifold $\CP^1 \times \dots \times \CP^1$ is divisible by $2$ and we 
can consider an integral symplectic form $\omega$ with $[\omega] = c_1 /2$. The corresponding prequantization 
is still a Gorenstein toric contact manifold, whose toric diagram is what we call in this paper a small cross-polytope.
As we will show in Section~\ref{sec:smallcrosspolytopes}, the small cross-polytope $S_n \subset \R^n$ is such that
\[
h^\ast_k (S_n) = \binom{n-1}{k}\,,\ k=0,\ldots,n-1\,, \ \text{and}\ h^\ast_n (S_n) = 0\,.
\]
This implies that the non-zero contact Betti numbers of this Gorenstein toric contact manifold 
$(N^{2n+1}_{S_n}, \xi_{S_n})$ are given by
\[
cb_{2k} (N^{2n+1}_{S_n}, \xi_{S_n}) = \sum_{j=0}^{k-1}  \binom{n-1}{j}\,,\ k=1, \ldots, n\,, \ \text{and} \ 
cb_{2k} (N^{2n+1}_{S_n}, \xi_{S_n}) = 2^{n-1}\,,\ \forall k>n\,.
\]
\begin{equation*}
\begin{array}{|c|c|c|c|c|c|c|cl} \hline
  \ast =  & \quad   0 \quad  & \quad 2 \quad & \quad 4 \quad & \quad \cdots \quad & \quad 2(n-1) \quad & \quad 2n \quad & \text{even}\, > 2n
\\ 
\hline 
cb_\ast (N^{2n+1}_{S_n}, \xi_{S_n}) &   0   & 1 & n & \cdots & 2^{n-1} -1 & 2^{n-1} &  2^{n-1}
\\
\hline
\end{array}
\end{equation*}
Theorem~\ref{thm:small}, proved in Section~\ref{sec:smallcrosspolytopes}, means that any other Gorenstein toric contact manifold 
of dimension $2n+1$ with these contact Betti numbers is equivariantly contactomorphic to $(N^{2n+1}_{S_n}, \xi_{S_n})$.

\subsection{Contact invariants of the primitive prequantization of $\CP(\xi_k \oplus \C) \to \CP^{n-1}$}
\label{ssec:primitive2}

When $0\leq k < n$ and $ k\equiv n \pmod 2$, we have that 
\[
M^{2n}_k = \CP (\xi_k \oplus \C)\,, \ \text{where}\ \xi_k \to \CP^{n-1} \ \text{is a $\C$-line bundle
with $c_1 (\xi_k) = k$,}
\]
is a monotone manifold with $c_1(M_k^{2n})$ divisible by $2$. In fact, up to translation, the moment
polytope of $(M^{2n}_k, [\omega_k] = c_1(M_k^{2n}))$ is the reflexive Delzant polytope $\Pp_k \subset \R^n$
\[
\mathcal{P}_k := \text{conv}\left( \{(n+k)\triangle_{n-1} - \bm{1}\} \times \{-1\}, \{(n-k)\triangle_{n-1} - \bm{1}\}  \times \{1\} \right)\,,
\]
with hyperplane description given by:
\begin{itemize}
\item $x_i +1 \geq 0, \forall i \in \{1, \dots, n \}$;
\item $- x_n + 1 \geq 0$;
\item $-\sum_{i=1}^{n-1} x_i- k x_n  +1 \geq 0$.
\end{itemize}
The length of its edges is $n+k$ on the bottom facet, $n-k$ on the top facet and $2$ on all lateral facets, i.e. the greatest
common divisor of the length of all edges is $2$ which implies that $c_1(M_k^{2n})/2$ is a primitive integral cohomology class
and $\Pp_1 / 2$ is the moment polytope of $(M^{2n}_k, [\omega] = c_1(M_k^{2n})/2)$.
See Figure~\ref{fig:bpsimplex} with $\Pp_1$ and $\Pp_1 / 2$ in dimension $n=3$.

\begin{figure}[]
\centering
\begin{minipage}{.5\textwidth}
	\centering
	\begin{tikzpicture}[scale=1.5, line join=bevel, tdplot_main_coords]
		\coordinate (e1) at (-1,-1,-1);
		\coordinate (e2) at (3,-1,-1);
		\coordinate (e3) at (-1,3,-1);
		\coordinate (m1) at (-1,-1,1);
		\coordinate (m2) at (1,-1,1);
		\coordinate (m3) at (-1,1,1);
		
		\draw [dashed] (e1) -- (e3);
		\draw [dashed] (e1) -- (e2);
		\draw (e2) -- (e3);
		\draw (m1) -- (m3);
		\draw (m1) -- (m2);
		\draw (m2) -- (m3);
		\draw [dashed] (e1) -- (m1);
		\draw (e2) -- (m2);
		\draw (e3) -- (m3);
		
		\draw [fill opacity=0.2,fill=black, draw=none] (e1) -- (e2) -- (e3) -- cycle;
		\draw [fill opacity=0.2,fill=black, draw=none] (m1) -- (m2) -- (m3) -- cycle;		
				
	\end{tikzpicture}
\end{minipage}%
\begin{minipage}{.5\textwidth}
	\centering
\begin{tikzpicture}[scale=1.5, line join=bevel, tdplot_main_coords]
		\coordinate (e1) at (0,0,0);
		\coordinate (e2) at (2,0,0);
		\coordinate (e3) at (0,2,0);
		\coordinate (m1) at (0,0,1);
		\coordinate (m2) at (1,0,1);
		\coordinate (m3) at (0,1,1);
		
		\draw [dashed] (e1) -- (e3);
		\draw [dashed] (e1) -- (e2);
		\draw (e2) -- (e3);
		\draw (m1) -- (m3);
		\draw (m1) -- (m2);
		\draw (m2) -- (m3);
		\draw [dashed] (e1) -- (m1);
		\draw (e2) -- (m2);
		\draw (e3) -- (m3);
		
		\draw [fill opacity=0.2,fill=black, draw=none] (e1) -- (e2) -- (e3) -- cycle;
		\draw [fill opacity=0.2,fill=black, draw=none] (m1) -- (m2) -- (m3) -- cycle;		
				
	\end{tikzpicture}
\end{minipage}

\caption{$\Pp_1$ and $\Pp_1 /2$ in dimension $n=3$.}
\label{fig:bpsimplex}
\end{figure}
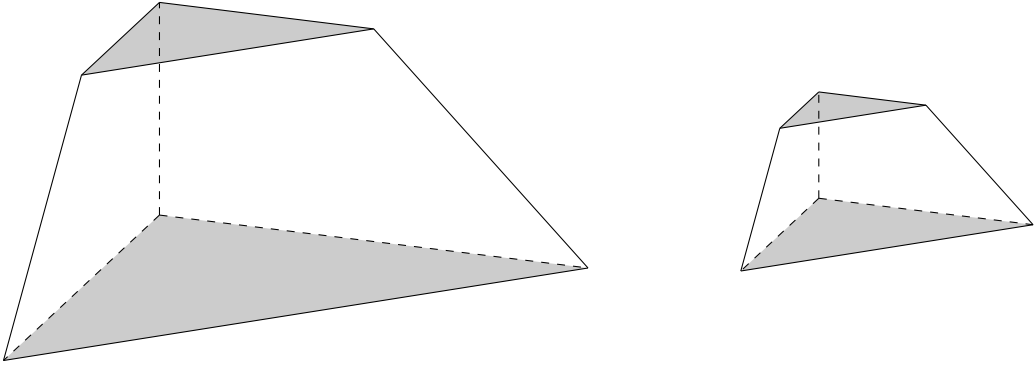

As we will show in Section~\ref{sec:bipysimplex}, the Gorenstein prequantization $(N^{2n+1}_k,\xi)$ of 
$(M^{2n}_k, \omega)$, with $[\omega] = c_1 (M_k^{2n}) /2$, has toric diagram $\Dd_k \subset \R^n$ with
\[
h^\ast_0 (\Dd_k) = \cdots = h^\ast_{n-1} (\Dd_k) =1 \  \text{and}\ h^\ast_n (\Dd_k) = 0\,.
\]
This implies that its contact Betti numbers are
\begin{equation*}
\begin{array}{|c|c|c|c|c|c|c|cl} \hline
  \ast =  & \quad   0 \quad  & \quad 2 \quad & \quad 4 \quad & \quad \cdots \quad & \quad 2(n-1) \quad & \quad 2n \quad & \text{even}\, > 2n
\\ 
\hline 
cb_\ast (N^{2n+1}_k) &   0   & 1 & 2 & \cdots & n-1 & n &  n
\\
\hline
\end{array}
\end{equation*}
Theorem~\ref{thm:bipysimplex1}, proved in Section~\ref{sec:bipysimplex}, means that any other Gorenstein toric contact manifold 
of dimension $2n+1$ with these contact Betti numbers is equivariantly contactomorphic to $(N^{2n+1}_k,\xi)$ for a unique $k\in\Z$ 
with $0\leq k < n$ and $k\equiv n \pmod 2$.

\section{Polytopes and Ehrhart theory}
\label{sec:polyEhrhart}

In this section we introduce the language of polytopes and Ehrhart theory. We also define the classes of polytopes 
that are most relevant for toric symplectic and contact geometry: Delzant polytopes, cones over them and toric 
diagrams. This section provides a quick review of combinatorial results that will be used in the paper. Proofs and 
technical details can be found in the indicated references.

\subsection{Polytopes, cones and polar duality}
\label{ssec:PCPD}

Here we give a quick introduction to polyhedra in $\mathbb{R}^n$. For a thorough presentation of this subject see 
for example~\cite{Z}.

\begin{definition}[Polytope]
A polytope $\mathcal{P} \subset \mathbb{R}^n$ is the convex hull of a finite set of points in $\mathbb{R}^n$, i.e.,
\begin{align*}
\mathcal{P} &=\conv \{\bm{v}_1, \bm{v}_2, \dots, \bm{v}_m\} \\
&=\left\{ \lambda_1 \bm{v}_1 + \lambda_2 \bm{v}_2 + \dots + \lambda_m \bm{v}_m : 
\sum_{i=1}^m \lambda_i = 1 \text{ and } \lambda_i \geq 0, i=1,\dots,m \right\}.
\end{align*}
This is commonly referred as the vertex or $\mathcal{V}$-representation of $\mathcal{P}$. Depending on the context, 
it might be better to deal with the hyperplane or $\mathcal{H}$-representation of a polytope in which we define a 
polytope by the intersection of a finite set of halfspaces delimited by hyperplanes, i.e.,
\begin{align*}
\mathcal{P} &=\{\bm{x} \in \mathbb{R}^n : A \bm{x} \leq  \bm{b}\} \\
&= \bigcap_{i=0}^d \left\{ \bm{x} \in \mathbb{R}^n : \bm{a}_i \cdot  \bm{x}   \leq b_i \right\}
\end{align*}
where $A \in \mathbb{R}^{d \times n}$ and $b \in \mathbb{R}^d$. This intersection is always a polyhedron but, 
since it might be unbounded, it is not always a polytope. A polytope is a bounded polyhedron.
	
All polytopes can be expressed in both representations, although the problem of algorithmically converting 
between them is non-trivial. 

The interior of a polytope $\mathcal{P}$ will be denoted by $\mathcal{P}^\circ$.
\end{definition}

We call a polytope \emph{integral} if all its vertices lie on the lattice $\mathbb{Z}^n$. 
An hyperplane 
\[
H=\{\bm{x} \in \mathbb{R}^n : \bm{a} \cdot \bm{x} =b\}
\]
is a supporting hyperplane of a polytope $\mathcal{P}$ if
\[
\mathcal{P} \subset \{\bm{x} \in \mathbb{R}^n : \bm{a} \cdot \bm{x} \leq b\} \text{ or } \mathcal{P} \subset \{\bm{x} \in \mathbb{R}^n : 
\bm{a} \cdot \bm{x} \geq b\}.
\]

\begin{definition}[Face]
A set $F \subset \mathbb{R}^n$ is face of a polytope $\mathcal{P}$ if
\begin{align*}
F=\mathcal{P} \cap \mathcal{H}
\end{align*}
for some supporting hyperplane $\mathcal{H}$.
\end{definition}

The dimension of a set $S \subset \mathbb{R}^n$ equals the dimension of the affine span generated by the points in $S$. 
If $S$ is a polytope and $\dim S=n$ we refer to $S$ as a $n$-polytope.

A $n$-polytope $\Pp$ has faces of every dimension between $-1$ and $n$, in particular:
\begin{itemize}
    \item $\mathcal{P}$ has a single face of dimension $n$, which is itself;
    \item $\mathcal{P}$ has faces of dimension $n-1$, called facets;
    \item $\mathcal{P}$ has faces of dimension $1$, called edges;
    \item $\mathcal{P}$ has faces of dimension $0$, called vertices;
    \item $\mathcal{P}$ has a single face of dimension $-1$, the $\emptyset$.
\end{itemize}
Throughout the text we will refer to the set of faces ordered by the (usual) inclusion as the face poset. 
We will say that two polytopes are combinatorially equivalent if their face posets are isomorphic.

A useful construction on a polytope is $\cone(\mathcal{P})$ that we define now.
\begin{definition}[Cone over a polytope]
For a polytope $\mathcal{P} \subset\R^n$
\[
\cone(\mathcal{P})=\{\bm{x} \in \mathbb{R}^n : \bm{x}=\alpha \bm{y}, \bm{y} \in \mathcal{P} \text{ and } \alpha \in \mathbb{R}_{\geq 0} \}.
\]	
\end{definition}

We define the polar dual of a polytope $\mathcal{P} \subset \mathbb{R}^n$.
\begin{definition}[Polar Duality]\label{def:polar}
For a polytope $\mathcal{P} \subset \mathbb{R}^n$ its polar dual
\begin{align*}
\mathcal{P}^*=\{ \bm{x} \in \mathbb{R}^n : \bm{x} \cdot \bm{y} \leq 1,  \forall \bm{y} \in \mathcal{P} \}
\end{align*}
is also a polytope.	
\end{definition}
\noindent If $0 \in \mathcal{P}$ then $\mathcal{P}^{**}=\mathcal{P}$. The polar dual has several interesting properties. 
Namely it suffices to check the inequalities in Definition~\ref{def:polar} for the vertices of $\mathcal{P}$ i.e. the facets of $\mathcal{P}^*$ 
are defined by the vertices of $\mathcal{P}$. This gives rise to a useful property of polar duality, it induces a bijection between 
faces of dimension $i$ in $\mathcal{P}$ and faces of dimension $n-i-1$ in $\mathcal{P}^*$. Stretching this reasoning a bit more, 
one can conclude that the face posets of $\mathcal{P}$ and $\mathcal{P}^*$ are anti-isomorphic.
\begin{definition}[Reflexive polytope]\label{def:reflexive}
An integral polytope $\mathcal{P}$ is said to be reflexive if $\mathcal{P}^*$ is also integral.
\end{definition}

Lastly, we cover the definition of simple and simplicial polytope.
\begin{definition}[Simple polytope]\label{def:simple}
A $n$-polytope is simple if every vertex is contained in $n$ edges.
\end{definition}
\begin{definition}[Simplicial polytope]\label{def:simplicial}
A $n$-polytope is simplicial if every facet is a $n-1$ simplex.
\end{definition}
\begin{lemma}
$\mathcal{P}$ is simple $\Leftrightarrow$ $\mathcal{P}^*$ is simplicial.
\end{lemma}

\subsection{Ehrhart theory}
\label{ssec:Ehrhart}

In this subsection we define the basics of Ehrhart theory that will be used in the paper. For a more detailed exposition 
and proofs see for example~\cite{BR}.

Ehrhart theory studies functions that count the number of integer points in the (integer) dilates of a polytope $\mathcal{P}$. 
Formally, we are interested in the function
\[
L_{\mathcal{P}}(t) := \#(t\mathcal{P} \cap \mathbb{Z}^n), \forall t \in \mathbb{Z}_{> 0}
\]
where $t\mathcal{P}=\{t \bm{x} \in \mathbb{R}^n: \bm{x} \in \mathcal{P}\}$.
\begin{theorem}[Ehrhart's Theorem]
If $\mathcal{P}$ is an integral $n$-polytope then $L_{\mathcal{P}}(t)$ is polynomial in $t$ with degree exactly $n$. 
\end{theorem}
\noindent From now on, we will refer to $L_\mathcal{P}$ as the Ehrhart polynomial of $\mathcal{P}$.

The Ehrhart series of a polytope is the generating function whose coefficients are given by the Ehrhart polynomial, i.e.,
\[
\Ehr_{\mathcal{P}}(z):=1+\sum_{t \geq 1} L_{\mathcal{P}}(t) z^t \,.
\]
In fact, we can rewrite this series as a rational function of a special form, which leads to the next theorem and corollary.
\begin{theorem}[Stanley's non negativity theorem]\label{thm:stanley}
For an integral $n$-polytope
\[
\Ehr_{\mathcal{P}}(z)=\frac{h^*_n z^n + h^*_{n-1} z^{n-1} + \dots + h^*_1 z + 1}{(1-z)^{n+1}}\,,
\]
where $h_1^*, h_2^*, \dots, h_n^*$ are non negative integers.
\end{theorem}
\begin{corollary}
\[
L_{\mathcal{P}}(t) = {t+n \choose n} + h_1^* {t+n-1 \choose n} + \dots + h_{n-1}^* {t+1 \choose n} + h_n^* {t \choose n}
\]
and so the coefficients $h_i^*$ are in fact the coefficients of the Ehrhart polynomial of $\mathcal{P}$ in the basis of the 
$n$-degree polynomials given by
\[
\left\{ {t+i\choose n}, i\in 0,\dots, n \right\}.
\]
\end{corollary}

The following theorem gives a particularly interesting characterization of reflexive polytopes.
\begin{theorem}\label{thm:reflexive}
The following conditions are equivalent:
\begin{enumerate}
\item $\mathcal{P}$ is reflexive;
\item $L_{\mathcal{P}^\circ} (t+1) = L_{\mathcal{P}} (t), \forall t\in \mathbb{Z}_{> 0}$, i.e., every lattice point 
is in $\partial k \mathcal{P}$ for a single $k \in \mathbb{Z}_{\geq 0}$;
\item $h_k^*=h^*_{n-k}$ for every $0 \leq k \leq \frac{n}{2}$.
\end{enumerate}
The fact that $(1) \Leftrightarrow (3)$ is often referred to as Hibi's palindromic theorem.
\end{theorem}

\subsection{Delzant polytopes, toric diagrams and prequantizations}
\label{ssec: Delzant}

Let us introduce two classes of polytopes that are relevant in the context of toric symplectic and contact geometry.
\begin{definition}[Delzant polytope]
A Delzant polytope is a simple $n$-polytope such that for every vertex the $n$ primitive edge vectors form a basis of 
$\mathbb{Z}^n$.
\end{definition}
\begin{definition}[Toric Diagrams] \label{def:diagram}
A toric diagram is a simplicial $n$-polytope such that the $n$ vertices of each facet can be translated to form a basis of 
$\mathbb{Z}^n$.
\end{definition}

We will now define the prequantization combinatorial construction that relates certain Delzant polytopes with toric diagrams.
\begin{definition}[Prequantization]\label{def:preq}
Start with a $n$-dimensional integral Delzant polytope $\mathcal{P} \subset \mathbb{R}^n$ defined by the $d$ hyperplanes 
$\bm{x} \cdot \bm{\nu}_i \geq -\lambda_i$, $\bm{\nu}_i \in \mathbb{Z}^n$, $\lambda_i \in \mathbb{Z}$ for $i \in 1, \dots, d$.
\begin{enumerate}
\item Consider the cone over $\mathcal{P} \times \{1\} \subset \R^n \times \R = \R^{n+1}$. This cone is defined by the hyperplanes 
$\bm{x} \cdot (\bm{\nu}_i , \lambda_i)\geq 0$, therefore for each facet of the cone the primitive inner normal vector is 
$(\bm{\nu}_i , \lambda_i) \in \mathbb{Z}^{n+1}$.
\item Apply a linear transformation $T\in GL(n, \Z)$ to all the normals such that $ T(\bm{\nu}_i, \lambda_i) = (\bm{v}_i, 1)$ where 
$\bm{v}_i \in \mathbb{Z}^{n}$.
\end{enumerate}
When such a $T$ exists, the $n$-polytope $\mathcal{D}=\text{conv}(\bm{v}_1, \dots, \bm{v}_{d})$ is a toric diagram. 
\end{definition}
\noindent The following theorem gives a necessary and sufficient condition for the existence of such a linear transformation $T$.
\begin{theorem}\label{thm:preq_are_gorenstein}
The transformation $T$ exists iff $r \mathcal{P}$ is reflexive for some $r \in \mathbb{Z}_{>0}$. In that case, 
$r \mathcal{D}$ is also reflexive.
\end{theorem}
\noindent We point to Section 7.3 in~\cite{HNP}, keeping in mind that although this text never refers to the prequantization this 
construction is tantamount to their dual cone.

From now on, we will refer to an integral polytope $\mathcal{P}$ such that $r\mathcal{P}$ is reflexive for some 
$r \in \mathbb{Z}_{>0}$ as a \emph{Gorenstein} polytope of index $r$. In fact, reflexive polytopes are Gorenstein of index 
$1$ and one is able to produce a theorem in the spirit of \autoref{thm:reflexive} but for Goresntein polytopes.
\begin{theorem}\label{thm:gorenstein}
The following conditions are equivalent:
\begin{enumerate}
\item $\mathcal{P}$ is Gorenstein of index $r$;
\item $L_{\mathcal{P}^\circ} (t) = L_{\mathcal{P}} (t-r), \forall t \in \mathbb{Z}_{>r}$ or 
$L_{\mathcal{P}} (t)= \#\{ \partial t \mathcal{P} \cap \mathbb{Z}^n \} + L_{\mathcal{P}} (t-r)$;
\item $h_k^*=h^*_{n+1-r-k}$ for every $0 \leq k \leq n+1-r$.
\end{enumerate}
\noindent Note that $(3)$ implies that the non zero coefficients of $h^*$ are still a palindrome even though their center of 
symmetry has changed.
\end{theorem}

We end this section with a basic example that illustrates what we discussed so far.
\begin{example}[Basic Example] \label{ex:hcube}
The (hyper-)cube $[-1, 1]^n$ is Delzant and reflexive for every $n$. The prequantization of this polytope 
has as toric diagram the cross-polytope 
\[
\diamond_n := \{(x_1, \ldots, x_n) \in \mathbb{R}^n : |x_1|+ \cdots +|x_n| \leq 1\} \quad\text{(see Figure~\ref{fig:0})}
\]
and we have that
\[
 L_{\diamond_n} (t) = \sum_{k=0}^{n} {n\choose k} {t-k+n \choose n}\,,\qquad
 \Ehr_{\diamond_n} (z) = \frac{(1+z)^n}{(1-z)^{n+1}} \,.
\]
\end{example}

\section{Small Cross-Polytopes}
\label{sec:smallcrosspolytopes}

This section is devoted to the prequantization of the unit cube: $\square_n := [0,1]^n$. 
This is related to our previous Example~\ref{ex:hcube}, the prequantization of twice the 
unit cube:
\[\begin{tikzcd}
	{[-1,1]^n=2\square_n} && {\diamond_n} \\
	\\
	{[0,1]^n=\square_n} && {S_n}
	\arrow["{\text{preq.}}", "{\text{- dual}}"', from=1-1, to=1-3]
	\arrow["{/2}", from=1-1, to=3-1]
	\arrow["{\text{preq.}}", from=3-1, to=3-3]
\end{tikzcd}\]
\noindent There are several relations between $\diamond_n$ and $S_n$ and in this section we outline 
some of them.

\subsection{Construction and basic properties} \label{ssec:smallcrosspolytopes}

We define the $n$-dimensional \emph{small cross-polytope} $S_n$ as the toric diagram of the prequantization 
of the unit cube $\square_n$. Note that this is possible since $2 \cdot \square_n$ is reflexive (and Delzant). 
The \emph{small} in the name comes from the fact that $\square_n = [0,1]^n$ is a small version of the 
(hyper-)cube $[-1, 1]^n$ whose prequantization has the cross-polytope $\diamond_n$ as toric diagram.

The $(n+1)$-dimensional cone over $[0,1]^n$ has the set of primitive inner lattice normal vectors
\[
\{ (\bm{e}_i, 0) \mid i \in 1,...,n \} \cup \{ (-\bm{e}_i, 1) \mid i \in 1,...,n \} \,.
\]
We then apply the unimodular (lattice preserving) transformation given by
\[
\begin{bmatrix}
	0 & 1 & 0 & \dots  & 0 \\
	0 & 0 & 1 & \ddots  & \vdots \\
	\vdots & \vdots & \ddots & \ddots & 0 \\
	0 & 0 & \dots & 0 & 1 \\
	1 & 1 & \dots & 1  & 2
\end{bmatrix}
\]
in order to map all the normals to the hyperplane $x_{n+1}=1$, concluding that the $2n$ vertices of a small cross-polytope are
\[
\{(0,...,0), (0,...,0,1) \} \cup \{ (\bm{e}_i, 0) \mid i \in 1,...,n-1 \} \cup \{ (-\bm{e}_i, 1) \mid i \in 1,...,n-1 \} \,.
\]
In particular, $S_n$ consists of the convex hull of a standard $(n-1)$-simplex in the $x_n=0$ hyperplane and the symmetric of 
the standard $(n-1)$-simplex in the $x_n=1$ hyperplane. These two simplices are shaded in Figure~\ref{fig:1}.
\begin{theorem}
Small cross-polytopes have the same face structure (face lattice) as cross-polytopes.
\end{theorem}
\begin{remark}{($\mathcal{H}$-representation)}
The vertex set of each facet has $n$ vertices, $k$ vertices of the top simplex and the complement $n-k$ vertices of the bottom simplex. 
The $\mathcal{H}$-representation of $t S_n$ is given by
\[
0 \leq \sum_{i \in A} x_i + x_n \leq t 
\]
where $A$ ranges over the $2^{n-1}$ subsets of $\{1,2,...,n-1\}$.
\end{remark}
\noindent A somewhat tedious manipulation of this $\mathcal{H}$-representation would enable us to conclude that
\begin{align*}
	L_{S_n}(t)&=\sum_{k=0}^{n-1} {n-1\choose k} \sum_{i=k}^{t} {t-i+n-k-1 \choose n-k-1} {i \choose k}\\
	 &=  \sum_{k=0}^{n-1} {n-1\choose k} {t-k+n \choose n}.
\end{align*}
We refrain from presenting this proof because it hides away the connection with cross-polytopes that seems 
to only appear in the Ehrhart polynomial of the very last step. This connection will become evident in the next 
two subsections, in which we prove the following (equivalent) theorem with two different geometric constructions.
\begin{theorem} \label{teo:2}
\[
\Ehr_{S_n} (z)=\frac{(1+z)^{n-1}}{(1-z)^{n+1}}\,.
\]
\end{theorem}

\subsection{Pyramids over cross-polytopes} \label{ssec:pyr}

By projecting $S_n$ into the hyperplane $x_n=0$ one gets a $(n-1)$-dimensional cross-polytope $\diamond_{n-1}$, 
see Figure~\ref{fig:2}. Furthermore, both the origin and the point $(0,...,0,1)$ get mapped to the origin, in fact, this is 
exactly what happens when projecting the pyramid over a cross-polytope $\Pyr (\diamond_{n-1})$. This is the motivation 
behind the following Lemma that generalizes this observation to integer dilations of these polytopes.
\begin{lemma} \label{teo:3}
If $\varphi: \mathbb{R}^n \longrightarrow \mathbb{R}^{n-1}$ is the projection defined by 
\[
\varphi(\bm{x})=(x_1, ..., x_{n-1})\,,
\]
then
\[
\#\{\bm{y} \in S_n \mid \bm{x}=\varphi (\bm{y}) \} = \#\{\bm{y} \in \emph{Pyr}(\diamond_{n-1}) \mid \bm{x}=\varphi (\bm{y}) \}
\]
holds for every lattice point $\bm{x} \in \mathbb{Z}^{n-1}$.
\end{lemma}
\begin{proof}[Proof of Theorem \ref{teo:2}]
Lemma~\ref{teo:3} implies that $\Ehr_{\Pyr (\diamond_{n-1})}=\Ehr_{S_n}$. Since 
\[
\Ehr_{\Pyr (\mathcal{P})} = \frac{1}{1-z} \Ehr (\mathcal{P}) \quad\text{and}\quad \Ehr_{\diamond_{n-1}}=\frac{(1+z)^{n-1}}{(1-z)^{n}}
\] 
(see for example Theorems 2.4 and 2.7 in~\cite{BR}),
Theorem~\ref{teo:2} follows directly. Lemma~\ref{teo:3} is, however, a stronger condition than just the equality of Ehrhart series.
\end{proof}

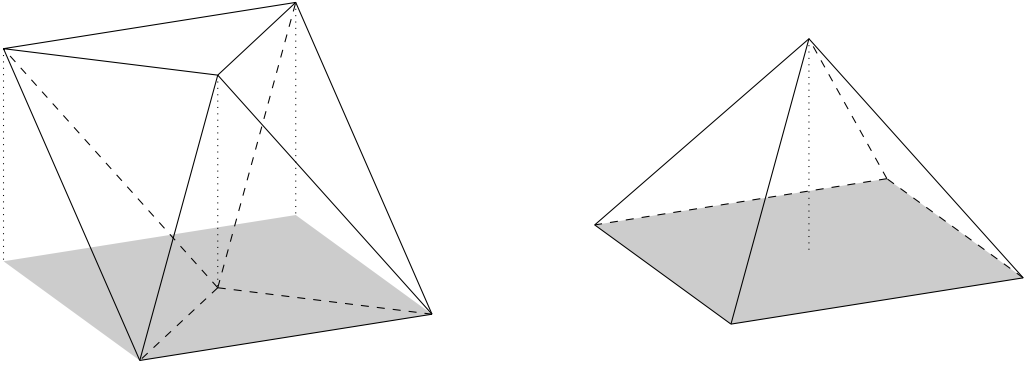
\begin{figure}[]
	\centering
	\begin{minipage}{.5\textwidth}
		\centering
		\begin{tikzpicture}[scale=3, line join=bevel, tdplot_main_coords]
			\coordinate (e1) at (0,0,0);
			\coordinate (e2) at (1,0,0);
			\coordinate (e3) at (0,1,0);
			\coordinate (m1) at (0,0,1);
			\coordinate (m2) at (-1,0,1);
			\coordinate (m3) at (0,-1,1);
			
			\coordinate (c1) at (-1, 0, 0);
			\coordinate (c2) at (0, -1, 0);
			
			\draw [fill opacity=0.2,fill=black, draw=none] (e2) -- (c2) -- (c1) -- (e3) -- cycle;
			\draw (m1) -- (m2) -- (m3) -- cycle;
			\draw [dashed] (e1) -- (m2);
			\draw [dashed] (e1) -- (m3);
			\draw (e2) -- (m1);
			\draw (e2) -- (m3);
			\draw (e3) -- (m1);
			\draw (e3) -- (m2);
			\draw [dashed] (e1) -- (e3);
			\draw [dashed] (e1) -- (e2);
			\draw (e2) -- (e3);
			
			\draw [dotted] (m2) -- (c1);
			\draw [dotted] (m3) -- (c2);
			\draw [dotted] (m1) -- (e1);
			
		\end{tikzpicture}
	\end{minipage}%
	\begin{minipage}{.5\textwidth}
		\centering
		\begin{tikzpicture}[scale=3, line join=bevel, tdplot_main_coords]
			\coordinate (e1) at (0,0,0);
			\coordinate (e2) at (1,0,0);
			\coordinate (e3) at (0,1,0);
			\coordinate (p) at (0,0,1);
			
			\coordinate (c1) at (-1, 0, 0);
			\coordinate (c2) at (0, -1, 0);
			
			\draw [fill opacity=0.2,fill=black,draw=none] (e2) -- (c2) -- (c1) -- (e3) -- cycle;
			\draw (e2) -- (c2);
			\draw [dashed] (c1) -- (c2);
			\draw (e2) -- (e3);
			\draw [dashed] (e3) -- (c1);
			\draw (e2) -- (p);
			\draw (e3) -- (p);
			\draw [dashed] (c1) -- (p);
			\draw (c2) -- (p);
			
			\draw [dotted] (p) -- (e1);
			
		\end{tikzpicture}
	\end{minipage}
	\caption{Projections of $S_3$ and $\Pyr (\diamond_2)$ to the hyperplane $x_d=0$.}
	\label{fig:2}
\end{figure}

\subsection{Pseudo-bipyramids over small cross-polytopes} \label{ssec:ppyr}

Cross-polytopes are often seen inductively as bipyramids over lower dimensional cross-polytopes. 
Small cross-polytopes share a similar property. The usual definition of a bypiramid is essentially the 
free sum with $[-1, 1]$, however, since $\bm{0} \in \partial S_n$, this would then imply that some facets 
of $\Bipyr (S_n)$ are no longer unimodular simplices, which in turn means it cannot be a toric diagram. 
This motivates the definition of a more general construction, namely the pseudo-bipyramids.
\begin{definition}{(Pseudo-bipyramid)} \label{def:ppyr}
Given a polytope $\mathcal{P}$ we define its pseudo-bipyramid as ${\Pbipyr} (\mathcal{P}, \bm{s}_1, \bm{s}_2)$ as
\[
\conv (\{(\bm{s}_1, 1), (\bm{s}_2, -1) \} \cup \{ (\bm{x}, 0) \mid \bm{x} \in \mathcal{P} \} ).
\]
where $\bm{s}_1, \bm{s}_2$ are lattice points of $\mathcal{P}$ (we will refer to them as special).
\end{definition}
\begin{lemma} \label{lemma:simplicial}
To create toric diagrams from other (lower dimensional) toric diagrams, $\bm{s}_1$ and $\bm{s}_2$ 
should not lie on the same facet of $\mathcal{P}$.
\end{lemma}
\begin{proof} 
If $\bm{s}_1$ and $\bm{s}_2$ were to lie on the same facet $F$ of $\mathcal{P}$ then the facet of 
$\text{P-bipyr}(\mathcal{P}, \bm{s}_1, \bm{s}_2)$ with vertices $(\bm{s}_1, 1), (\bm{s}_2, -1)$ and 
$ \{(\bm{x},0) \mid \bm{x} \in V(F)\}$ would not be simplicial since it would have $n+3$ vertices and 
$\text{P-bipyr}(\mathcal{P}, \bm{s}_1, \bm{s}_2)$ is of dimension $n+1$.
\end{proof}
\noindent From an Ehrhart theory perspective, this construction is equivalent to the usual bipyramids: we are essentially 
joining two pyramids that share the same base. This yields the following theorem.
\begin{theorem} \label{teo:6}
$\Ehr_{\Pbipyr (\mathcal{P}, \bm{s}_1, \bm{s}_2)}=\frac{1+z}{1-z} \Ehr_\mathcal{P}$
\end{theorem}

By projecting $S_n$ to the hyperplane $x_1=0$ we get $S_{n-1}$. Figure~\ref{fig:3} shows this construction with 
the special vertices highlighted. We should also notice that instead of $x_1=0$ we could have chosen any of the 
hyperplanes $x_2=0$, ..., $x_{n-1}=0$.
\begin{lemma}
The small cross-polytope $S_n$ (up to reordering coordinates) can be constructed as
\[
\Pbipyr (S_{n-1}, (0,...,0), (0,...,0,1))
\]
\end{lemma}
\begin{remark} \label{remark:2}
The choice of special vertices when constructing $S_n$ from $S_{n-1}$ does not really matter 
in the case of small cross-polytopes. There are only $n-1$ choices of pairs that do not share a 
facet and all of them yield the same polytope (up to unimodular equivalence).
\end{remark}
This construction not only proves Theorem~\ref{teo:2} once again but also enables us to effortlessly 
prove various inductive properties. The following theorem is an example.
\begin{theorem} \label{teo:4}
Small cross-polytopes admit a unimodular triangulation.
\end{theorem}
\begin{proof}
By induction, the base case is trivial, for the inductive step note that for any $(n-1)$-simplex $\Delta$ in the unimodular 
triangulation of $S_{n-1}$ then $\conv (\Delta, (\bm{0},1))$ and $\conv (\Delta, (\bm{0},1,-1))$ are unimodular $n$-simplices. 
Taking these simplices over all $\Delta$ forms a unimodular triangulation of $S_n$.
\end{proof}

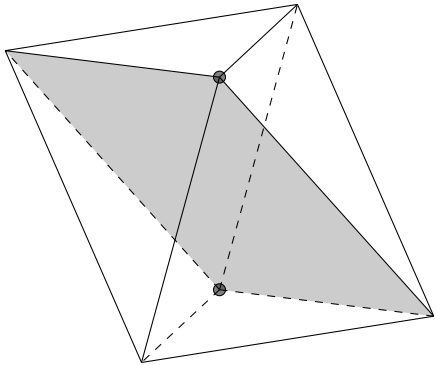
\begin{figure}[]
	\centering
	\begin{tikzpicture}[scale=3, line join=bevel, tdplot_main_coords]
		\coordinate (e1) at (0,0,0);
		\coordinate (e2) at (1,0,0);
		\coordinate (e3) at (0,1,0);
		\coordinate (m1) at (0,0,1);
		\coordinate (m2) at (-1,0,1);
		\coordinate (m3) at (0,-1,1);
		
		\draw [fill opacity=0.2,fill=black, draw=none] (e1) -- (m3) -- (m1) -- (e3) -- cycle;
		\draw (m1) -- (m2) -- (m3) -- cycle;
		\draw [dashed] (e1) -- (m2);
		\draw [dashed] (e1) -- (m3);
		\draw (e2) -- (m1);
		\draw (e2) -- (m3);
		\draw (e3) -- (m1);
		\draw (e3) -- (m2);
		\draw [dashed] (e1) -- (e3);
		\draw [dashed] (e1) -- (e2);
		\draw (e2) -- (e3);
		
		\draw [fill opacity=0.5,fill=black] (e1) circle (0.75pt);
		\draw [fill opacity=0.5,fill=black] (m1) circle (0.75pt);
		
	\end{tikzpicture}
	\caption{Projection of $S_3$ to the hyperplane $x_1=0$.}
	\label{fig:3}
\end{figure}

\subsection{Small cross-polytopes are Ehrhart determined} \label{ssec:determined}

\begin{theorem}
Let $\mathcal{D}$ be a toric diagram. If the Ehrhart polynomial of $\mathcal{D}$ has $h^*$-polynomial 
$h^*_\mathcal{D}(z) = (1+z)^{n-1}$, then $\mathcal{D}$ is unimodularly equivalent to the small cross polytope.
\end{theorem}
\begin{proof}
The degree $n-1$ of $h^*_\mathcal{D}(z)$ indicates that $\mathcal{D}$ is Gorenstein of index $2$. Thus, the 
dilated polytope $2\mathcal{D}$ contains a unique strictly interior lattice point $c \in \mathbb{Z}^n$. 
		
Evaluating the Ehrhart polynomial at $t=1$ gives the total number of lattice points in $\mathcal{D}$. 
Because $\mathcal{D}$ is Gorenstein of index $2$, its interior is empty of lattice points. Furthermore, 
since $\mathcal{D}$ is a toric diagram, its proper faces are unimodular simplices, meaning they contain 
no boundary lattice points other than their actual vertices. Thus, the lattice points of $\mathcal{D}$ are exactly 
its vertices, giving $v = h^*_1 + d + 1 = 2d$ vertices.
		
Next, we count the lattice points in $2\mathcal{D}$. For a unimodular $k$-simplex $\Delta$, the number of strictly 
interior lattice points in its $t$-th dilation is given by the standard discrete volume formula 
$L_{\Delta^\circ}(t) = \binom{t-1}{k}$. At a dilation of $t=2$, any face of dimension $k \ge 2$ contains exactly 
$\binom{1}{k} = 0$ interior points. Therefore, the lattice points in $2\mathcal{D}$ consist solely of the center $\bm{c}$, 
the $2n$ vertices, and the strict midpoints of its $e$ edges.  Evaluating the Ehrhart polynomial at $t=2$ gives
\[ 
L_{\mathcal{D}}(2) = \sum_{i=0}^n \binom{n-1}{i} \binom{n+2-i}{n} = 2n^2 + 1 \ \text{lattice points.}
\]
Equating this to our structural geometric count $1 + 2n + e$ yields exactly $e = 2n(n-1)$ edges.
		
Define the shifted polytope $\mathcal{T} := 2\mathcal{D} - \bm{c}$, which places the unique interior point at the origin 
$\mathbf{0}$. Its vertices are $\bm{v'}_i = 2\bm{v}_i - \bm{c}$, implying $\bm{v'}_i \equiv - \bm{c} \pmod 2$. Because 
all $2d$ vertices of $\mathcal{T}$ share the same parity, the midpoint of any pair $\bm{u}, \bm{w} \in \text{Vert}(\mathcal{T})$ 
is an integer lattice point. 
		
If $\bm{u}$ and $\bm{w}$ do not form an edge in $\mathcal{T}$, their connecting segment crosses the interior. 
Since $\bm{0}$ is the only interior lattice point, their midpoint must be $\bm{0}$, forcing $\bm{w} = - \bm{u}$. 
		
We now look at the $1$-skeleton (the graph of vertices and edges) of $\mathcal{T}$. A complete graph on 
$2n$ vertices has $\binom{2n}{2} = 2n^2 - n$ edges. Since $\mathcal{T}$ has $2n(n-1) = 2n^2 - 2n$ edges, 
it is missing exactly $n$ edges (non-adjacent pairs). Since a vertex can have at most one geometric antipode 
crossing the origin, these $n$ missing edges must be strictly disjoint, perfectly pairing all $2n$ vertices into 
geometric opposites.
		
We now translate this central symmetry back to $\mathcal{D}$. The opposite vertices in $\mathcal{T}$ satisfy 
$\bm{w'}_i = - \bm{v'}_i$. Substituting our coordinate shift, we get $(2\bm{w}_i - \bm{c}) = -(2\bm{v}_i - \bm{c})$. 
Rearranging this gives $\bm{v}_i + \bm{w}_i = \bm{c}$. Geometrically, this confirms that the midpoint of any 
opposite pair $(\bm{v}_i, \bm{w}_i)$ in $\mathcal{D}$ is exactly $\frac{\bm{c}}{2}$. Thus, the $2n$ vertices of 
$\mathcal{D}$ consist of $n$ pairs symmetrically arranged around the half-integer center $\frac{\bm{c}}{2}$.
		
Let $F$ be a facet of $\mathcal{D}$ with vertices $\bm{v}_1, \dots, \bm{v}_n$. Let $\bm{e}_1, \dots, \bm{e}_n$ 
denote the standard basis of $\mathbb{R}^n$. Because $F$ is a unimodular simplex, an affine unimodular 
transformation maps $F$ onto the hyperplane $x_d = 0$ such that its vertices are 
$\bm{v}_i = \bm{e}_i$ for $1 \le i \le n-1$, and $\bm{v}_n = \bm{0}$. 
		
Since the lattice distance from $\bm{c}$ to $2F$ is exactly $1$, the $n$-th coordinate of $\bm{c}$ must be $\pm 1$. 
After possibly applying a reflection across $x_n=0$ we can guarantee it is $1$. We apply a further unimodular 
shear parallel to $x_n = 0$, which preserves $F$ and maps $\bm{c}$ exactly to the $n$-th basis vector $\bm{e}_n$. 
		
The remaining $n$ vertices of $\mathcal{D}$ are now rigidly determined by the symmetry $\bm{w}_i = \bm{c} - \bm{v}_i$. 
For $1 \le i \le n-1$, we obtain $\bm{w}_i = \bm{e}_n - \bm{e}_i$, and for the origin, $\bm{w}_n = \bm{e}_n - \bm{0} = \bm{e}_n$. 
This exact configuration of vertices, the basis vectors $\bm{e}_1, \dots, \bm{e}_{n-1}$ and origin $\bm{0}$, alongside 
their shifted symmetrics $\bm{e}_n - \bm{e}_1, \dots, \bm{e}_n - \bm{e}_{n-1}$ and $\bm{e}_n$ is precisely the standard 
construction of the small cross polytope, proving uniqueness up to unimodular equivalence.
\end{proof}

\subsection{Final remarks}  \label{ssec:remarks}

The Ehrhart polynomial of a cross-polytope satisfies two curious properties: (i) it is symmetric in
$n$ and $t$, and (ii) all its roots have real part equal to $-1/2$ \cite{BCKV, KPT}. The Ehrhart polynomial 
of a small cross-polytope is not symmetric in $n$ and $t$, but is does satisfy an interesting property 
analogous to (ii).
\begin{theorem}
All the roots of the Ehrhart polynomial of a small cross-polytope have real part equal to $-1$.
\end{theorem}
\begin{proof} Due to Theorem~\ref{teo:2}, the main result in~\cite{R-V} implies this result.
\end{proof}

%
%
%

In~\cite{BHW} the following theorem is proved:
\begin{theorem}
If for some integral polytope $\mathcal{P}$ all the complex roots of $L_\mathcal{P}$ have real part $-1/2$, then $\mathcal{P}$ is reflexive.
\end{theorem}
\noindent Can we infer a similar property for the polytopes such that all the roots have real part $-1$?

\section{On prequantizations of Gorenstein Delzant polytopes}
\label{sec:monotone}

In this section we consider a construction similar to the one in the previous one, but now for a 
general reflexive Delzant $n$-polytope $\mathcal{P}$ such that $\mathcal{P}/r$ is integral for some $r \in \mathbb{Z}_{>0}$.
\[\begin{tikzcd}
	{\mathcal{P}} && {\mathcal{D}} \\
	\\
	{\mathcal{P}/r} && {\mathcal{D}'}
	\arrow["{\text{preq.}}", "{\text{- dual}}"', from=1-1, to=1-3]
	\arrow["{/r}", from=1-1, to=3-1]
	\arrow["{\text{preq.}}", from=3-1, to=3-3]
\end{tikzcd}\]
Our main result in this context is the following theorem.
\begin{theorem} \label{thm:main}
\[
\Ehr_{\mathcal{D}}(z)=\frac{1-z^r}{1-z}\Ehr_{\mathcal{D}'}(z) \,.
\]
\end{theorem}
\begin{corollary}\label{cor:1}
Let $\mathcal{P}$ be a Gorenstein polytope of index $s$ such that $\mathcal{P}/r$ is integral for some 
$r \in \mathbb{Z}_{>0}$. Let $\mathcal{D}$ be the toric diagram of the prequantization of $\mathcal{P}$ and 
$\mathcal{D}'$ the one of $\mathcal{P}/r$. We then have that
\[
\Ehr_{\mathcal{D}}(z)=\frac{1-z^{rs}}{1-z^s}\Ehr_{\mathcal{D}'}(z).
\]
\end{corollary}
\begin{proof}{(Corollary)}
Let $\mathcal{T}$ be the toric diagram of the prequantization of the reflexive polytope 
$s\mathcal{P}=sr(\mathcal{P}/r)$. Applying \autoref{thm:main} twice, once to $\mathcal{P}$ 
and once to $\mathcal{P}/r$, we have that
\[
\frac{1-z^s}{1-z}\Ehr_{\mathcal{D}}(z)=\Ehr_{\mathcal{T}}(z)=\frac{1-z^{rs}}{1-z}\Ehr_{\mathcal{D}'}(z)
\]
which proves the corollary. Note that $\mathcal{D}'$ is Gorenstein of index $rs$.
\end{proof}
\begin{remark} \label{rem:main}
It is interesting, although direct, to check that the $h^*$ polynomial of $\mathcal{D}$ is exactly the sum of $r$ copies of the $h^*$ polynomial of $\mathcal{D}'$ such that each one has a shift of $s$ indices to the right. 
More precisely,
\[
h^*_{\mathcal{D}}(z)=\left(1+z^s+z^{2s}+\dots+z^{s(r-1)}\right)h^*_{\mathcal{D}'}(z)\,.
\]
Furthermore, we note that the case $s=1$ is of particular interest. In this case, we have a reflexive toric diagram $\mathcal{D}$ 
and toric diagram $\mathcal{D}'$ that is Gorenstein of index $r$. The previous expression can be stated as
\[
\sum_{j=x-r+1}^{x} h_j^*(\mathcal{D}')= h_x^*(\mathcal{D})\,.
\]
We assume throughout that $h_x^*(\mathcal{D}')= h_x^*(\mathcal{D})=0$ for $x<0$ and for $x>n$. Taking consecutive differences 
yields
\[
h_{x-r}^*(\mathcal{D}') - h_x^*(\mathcal{D}') = h_{x-1}^*(\mathcal{D}) - h_x^*(\mathcal{D})\,.
\]
Let $x = n-m+rk+r$ and apply the palindromic property $h_y^*(\mathcal{D}) = h_{n-y}^*(\mathcal{D})$ to the right hand side to get
\[
h_{n-m+rk}^*(\mathcal{D}') - h_{n-m+r(k+1)}^*(\mathcal{D}') = h_{m-rk-r+1}^*(\mathcal{D}) - h_{m-rk-r}^*(\mathcal{D})\,.
\]
Summing both sides over $k \ge 0$ we have
\[
h_{n-m}^*(\mathcal{D}') = \sum_{k \ge 0} \left( h_{m-rk-r+1}^*(\mathcal{D}) - h_{m-rk-r}^*(\mathcal{D}) \right)\,,
\]
recognizing the left hand side as a telescoping sum. In terms of contact Betti numbers, and
since $cb_{2i}(N_{\mathcal{D}'}, \xi_{\mathcal{D}'})=0$ for all $i<0$, we have that
\[
cb_{2i}(N_{\mathcal{D}'}, \xi_{\mathcal{D}'}) = \sum_{m=0}^i (cb_{2m}(N_{\mathcal{D}'}, \xi_{\mathcal{D}'}) - 
cb_{2(m-1)}(N_{\mathcal{D}'}, \xi_{\mathcal{D}'})) = \sum_{m=0}^i h_{n-m}^*(\mathcal{D}')\,.
\]
Substituting the expression for $h_{n-m}^*(\mathcal{D}')$, exchanging the order of summation, and evaluating the inner 
telescoping sum, we get
\begin{align*}
cb_{2i}(N_{\mathcal{D}'}, \xi_{\mathcal{D}'}) &= \sum_{m=0}^i \sum_{k \ge 0} 
\left( h_{m-rk-r+1}^*(\mathcal{D}) - h_{m-rk-r}^*(\mathcal{D}) \right) \\
&= \sum_{k \ge 0} h_{i-rk-r+1}^*(\mathcal{D})\,.
\end{align*}
Recall that $cb_{2i}(N_{\mathcal{D}'}, \xi_{\mathcal{D}'})=\rank HC_{2i}(N_{\mathcal{D}'}, \xi_{\mathcal{D}'})$ 
and $h_{i}^*(\mathcal{D})=\dim H^{2i}(M_\Pp; \mathbb{Q})$, where $M_\Pp$ is the toric symplectic manifold with 
reflexive Delzant polytope $\Pp = - \Dd^\ast$ (cf. equation~(\ref{eq:Betti})). Thus, in this toric context, 
Theorem~\ref{thm:main} is the Ehrhart theory version of 
\[
HC_{2i}(N_{\mathcal{D}'}, \xi_{\mathcal{D}'} )= \bigoplus_{k \geq 0} H^{2(i-rk-r+1)}(M_\mathcal{P}; \mathbb{Q})\,,
\]
which is exactly Corollary~1.9 in~\cite{AMM} and also follows from a more general result in~\cite{B}.
\end{remark}

\subsection{Proof of Theorem~\ref{thm:main}}

Let $\mathcal{P}$ be a reflexive Delzant polytope such that $\mathcal{P}/r$ is integral. 
Let us determine the toric diagrams of the prequantization of  $\mathcal{P}$ and $\mathcal{P}/r$.

Translate $\mathcal{P}$ by a lattice vector such that $\bm{0} \in \mathcal{V}(\mathcal{P})$. 
For the sake of this proof, we will only perform the first step of Definition~\ref{def:preq}, letting 
our toric diagram (a $n$-polytope) remain contained in a hyperplane of $\mathbb{R}^{n+1}$. 
This clearly does not affect the integer points contained in the polytope compared to the final 
toric diagram, since the only missing step is a unimodular transformation.

If $\mathcal{P}$ is defined by the $d$ hyperplanes $\bm{x} \cdot \bm{\nu}_i \geq -\lambda_i$, with 
$\bm{\nu}_i \in \mathbb{Z}^n$ and $\lambda_i \in \mathbb{Z}$ for $i = 1, \dots, d$, then the 
vertices of $\mathcal{D}$ are $(\bm{\nu}_i, \lambda_i)$. Furthermore, since $\mathcal{P}/r$ is 
defined by the hyperplanes $\bm{x} \cdot \bm{\nu}_i \geq -\lambda_i/r$, the vertices of $\mathcal{D}'$ 
are $(\bm{\nu}_i, \lambda_i/r)$, which are integer by hypothesis.

Note that $\bm{\nu}_i$ are the vertices of the projection of both these polytopes onto the hyperplane 
normal to $(0,\dots,0,1)$. These vertices are the vertices of $-\mathcal{P}^\ast$, where $\Pp^\ast$ is
the polar dual of $\mathcal{P}$. In fact, it is easy to check the following lemma.
\begin{lemma}\label{lemma:0}
The face poset (or combinatorial structure) of $\mathcal{D}$ coincides with the face poset of 
$\mathcal{D}'$.
\end{lemma}
\begin{proof}
From the previous paragraph, we essentially have the "same" set of vertices but located in two 
different hyperplanes of $\mathbb{R}^{n+1}$. More precisely, consider the linear map 
$\Phi: \mathbb{R}^{n+1} \to \mathbb{R}^{n+1}$ defined by the diagonal matrix 
$\text{diag}(1, \dots, 1, 1/r)$. This map is a linear isomorphism that sends $(\bm{\nu}_i, \lambda_i)$ to 
$(\bm{\nu}_i, \lambda_i/r)$ for every $i \in 1, \dots, d$, implying that $\Phi(\mathcal{D}) = \mathcal{D}'$. 
Since linear isomorphisms preserve the face structure of polytopes, $\mathcal{D}$ and $\mathcal{D}'$ 
are combinatorially equivalent.
\end{proof}
\noindent This lemma, together with the fact that $\mathcal{D}$ and $\mathcal{D}'$ are toric diagrams, 
implies that their boundaries will always be Ehrhart-equivalent. This is formalized in the next lemma.

\begin{lemma}\label{lemma:1}
$\#\{\partial t \mathcal{D} \cap \mathbb{Z}^n \} = \#\{\partial t \mathcal{D}' \cap \mathbb{Z}^n \}$ 
for all $t \in \mathbb{Z}_{>0}$.
\end{lemma}
\begin{proof}
From Lemma~\ref{lemma:0}, one can conclude that the number of $i$-dimensional faces,
$i=-1, \dots, n$, of $\mathcal{D}$ and $\mathcal{D}'$ coincide. Recalling that $\mathcal{D}$ and $\mathcal{D}'$ 
are toric diagrams, we know that all their faces of dimension $\leq n-1$ are unimodular simplices.
This implies that there are unimodular triangulations (with no new vertices) of $\partial \mathcal{D}$ 
and $\partial \mathcal{D}'$. Since the polytopes share the same face lattice and every corresponding 
face is a unimodular simplex, the Principle of Inclusion-Exclusion yields an identical number of boundary 
lattice points for both $\mathcal{D}$ and $\mathcal{D}'$.
\end{proof}

\begin{proof}[Proof of Theorem~\ref{thm:main}]
It is easy to see that Theorem~\ref{thm:main} is equivalent to
\[
L_{\mathcal{D}}(t)=\sum_{i=0}^{r-1} L_{\mathcal{D}'}(t-i)
\]
so we can finish the proof by induction in $t$. The base case is trivial. For the inductive step we use
the fact that $D$ is reflexive, so we can apply Theorem~\ref{thm:reflexive} to get

\begin{align*}
	L_{\mathcal{D}}(t+1)
	&=\#\{\partial (t+1) \mathcal{D} \cap \mathbb{Z}^n \}+L_{\mathcal{D}}(t) \\
	&=\#\{\partial (t+1) \mathcal{D} \cap \mathbb{Z}^n \}+\sum_{i=0}^{r-1}  L_{\mathcal{D}'}(t-i) 
	\tag{induction step} \\
	&=\#\{\partial (t+1) \mathcal{D}' \cap \mathbb{Z}^n \}+\sum_{i=0}^{r-1} L_{\mathcal{D}'}(t-i)	
	\tag{Lemma~\ref{lemma:1}} \\
	&=L_{\mathcal{D}'}(t+1)+\sum_{i=0}^{r-2} L_{\mathcal{D}'}(t-i) 
	\tag{$\mathcal{D}'$ is Gorenstein of index $r$ + Theorem~\ref{thm:gorenstein}} \\
	&=\sum_{i=0}^{r-1} L_{\mathcal{D}'}(t+1-i) \,.
\end{align*}
\end{proof}

\section{A class of toric diagrams Ehrhart equivalent to cross-polytopes} \label{sec:ppyr}

In this section we construct and classify a class of toric diagrams that share the Ehrhart series of the 
cross-polytope $\diamond_d$. This class consists of the toric diagrams that arise from the 
prequantization of monotone Bott manifolds (cf. Section~\ref{sec:FBm}, in particular Theorem~\ref{thm:S}).

\subsection{Iterative construction via pseudo-bipyramids}
Recall that pseudo-bipyramids where defined in Definition~\ref{def:ppyr}.
Denoting the set of $n$-dimensional polytopes in our class by $\mathcal{C}_n$, we define the next iteration as
\[
\mathcal{C}_{n+1}=\{ \text{Pbipyr} (\mathcal{P}, \bm{s}_1, \bm{s}_2) \mid \mathcal{P} \in \mathcal{C}_{n}, 
\bm{s}_1, \bm{s}_2 \in \mathbb{Z}^n \cap \mathcal{P} \}.
\]
and initialize $\mathcal{C}_1$ with just one element, the toric diagram $[-1,1]$. For any polytope 
$\mathcal{P} \in \mathcal{C}_n$, because it is Ehrhart equivalent to $\diamond_n$, the only lattice points 
it contains are its $2n$ boundary vertices and the origin as its unique interior point. Thus, when choosing 
the points $\bm{s}_1, \bm{s}_2$ for the next iteration, we are restricted to three cases:
\begin{itemize}
\item $\bm{s}_1=\bm{s}_2=\bm{0}$, which yields the standard bipyramid;
\item $\bm{s}_1=\bm{0}$ and $\bm{s}_2 \in \partial \mathcal{P} $;
\item $\bm{s}_1, \bm{s}_2 \in \partial \mathcal{P} $.
\end{itemize}
\begin{lemma}
The third case ($\bm{s}_1, \bm{s}_2 \in \partial \mathcal{P}$) is redundant, as it always produces a polytope 
unimodularly equivalent to one generated by one of the other cases.
\end{lemma}
\begin{proof}
Let $\Qq = \text{P-bipyr}(\mathcal{P}, \bm{s}_1, \bm{s}_2)$ where $\bm{s}_1, \bm{s}_2 \in \partial \mathcal{P}$. 
We define a mapping $\varphi: \mathbb{R}^{n+1} \to \mathbb{R}^{n+1}$ given by the transformation
\[
\varphi(\bm{x}, x_{n+1}) = (\bm{x} - x_{n+1} \bm{s}_1, x_{n+1})\,.
\]
Because $\bm{s}_1 \in \mathbb{Z}^n$, $\varphi$ is a unimodular transformation. Applying $\varphi$ to the vertices 
of $\Qq$ yields
\begin{itemize}
\item $\varphi(\bm{x}, 0) = (\bm{x}, 0)$ for all $\bm{x} \in \Vv (\mathcal{P})$;
\item $\varphi(\bm{s}_1, 1) = (\bm{s}_1 - \bm{s}_1, 1) = (\bm{0}, 1)$;
\item $\varphi(\bm{s}_2, -1) = (\bm{s}_2 + \bm{s}_1, -1)$.
\end{itemize}
Therefore, $\varphi(\Qq)$ has vertices equivalent to a pseudo-bipyramid with special points at the origin and at 
$\bm{y} = \bm{s}_1 + \bm{s}_2$. We must verify that $\bm{y}$ is a valid lattice point in $\mathcal{P}$. 
Because $\mathcal{P}$ is Ehrhart equivalent to $\diamond_n$ and its unique interior lattice point is the origin $\bm{0}$, 
we know that $\mathcal{P}$ is a reflexive polytope. Therefore $\mathcal{P}$ can be defined as the intersection of 
half-spaces $\bm{x}\cdot\bm{\nu}_F +1 \geq 0$, where  $\bm{\nu}_F \in \mathbb{Z}^n$ is the primitive interior normal 
to the facet $F$. Since $\bm{s}_1, \bm{s}_2 \in \mathbb{Z}^n \cap \mathcal{P}$, their dot products with $\bm{\nu}_F$ 
are integers lower bounded by $-1$. By Lemma~\ref{lemma:simplicial}, to ensure that the resulting pseudo-bipyramid 
is simplicial, the special points $\bm{s}_1$ and $\bm{s}_2$ cannot lie on the same facet of $\mathcal{P}$. Therefore, 
they cannot simultaneously satisfy $\bm{x}\cdot\bm{\nu}_F = -1$ for any given facet $F$. Because the dot products 
are integers, for every facet $F$ at least one of $\bm{s}_1 \cdot \bm{\nu}_F$ or $\bm{s}_2 \cdot \bm{\nu}_F$ must be
strictly greater than $-1$, meaning it must be $\geq 0$. By linearity of the inner product,  
$(\bm{s}_1 + \bm{s}_2)\cdot \bm{\nu}_F + 1 \geq 0$ for all facets $F$. This proves that $\bm{y} = \bm{s}_1 + \bm{s}_2$ 
satisfies all facet inequalities and therefore is in $\mathcal{P}$. Since $\bm{y}$ is the sum of two integer vectors, it is 
a lattice point in $\mathcal{P}$, hence it is either $\bm{0}$ or belongs to the boundary $\partial \mathcal{P}$. 
Thus, $\varphi(\Qq) = \text{P-bipyr}(\mathcal{P}, \bm{0}, \bm{y})$ with $\bm{y} = \bm{0}$ or $\bm{y} \in \partial \mathcal{P}$
as desired.
\end{proof}

\begin{theorem} \label{thm:pbipyr}
The vertex set of any polytope in $\mathcal{C}_{n+1}$ is of the form
\[
\{(0,...,0, 1), (\bm{y}, -1) \} \cup \{ (\bm{x}, 0) \mid \bm{x} \in \mathcal{P} \}\,,
\]
for some $\mathcal{P} \in \mathcal{C}_{n}$ and $\bm{y} \in \mathbb{Z}^n \cap \mathcal{P}$.
\end{theorem}
\begin{proof}
By the preceding Lemma, we only need to consider constructions where at least one of the new points is the origin.
\end{proof}
\noindent While this characterizes the coordinates perfectly, enumerating the family remains problematic since 
different choices of $\mathcal{P}$ or $\bm{y}$ can yield toric diagrams that are unimodularly equivalent.

\subsection{Classification via Rooted $3$-Cacti}

To systematically classify the toric diagrams in $\mathcal{C}_n$ up to unimodular equivalence, we encode the 
iterative pseudo-bipyramid construction into a purely combinatorial object, a rooted $3$-cactus graph. There 
are several possible choices of  combinatorial objects to carry out this classification. We chose rooted $3$-cacti 
because they have already been related to monotone Bott manifolds in~\cite{CLMP-2}.

A $3$-cactus (or triangle-cactus) is a connected graph in which every edge belongs to exactly one simple cycle, 
and every such cycle is a triangle ($K_3$). We can naturally construct a bijection between the lattice points in 
a polytope $\mathcal{P} \in \mathcal{C}_n$ and the vertices of a rooted $3$-cactus $T$, where the root 
corresponds to the origin $\bm{0}$, and every other node corresponds to a non-zero boundary lattice point of 
the polytope. Note that whenever a rooted $3$-cactus has $n$ triangles then it has $3n-(n-1)=2n+1$ vertices.

We define this mapping explicitly by building the graph incrementally, alongside the pseudo-bipyramid construction. 
For the base case $n=1$, we start with the polytope $\mathcal{P}_1 = [-1, 1] \in \mathcal{C}_1$. It contains exactly 
three lattice points: the origin $0$, and the boundary points $-1$ and $1$. The corresponding graph $G_1$ is 
initialized as a single triangle ($K_3$), where the root node represents $0$ and the other two vertices represent $-1$ 
and $1$. Suppose we have a $3$-cactus $G_k$ representing a polytope $\mathcal{P}_{k} \in \mathcal{C}_{k}$. 
To construct the next iteration $\mathcal{P}_{k+1} = \text{P-bipyr}(\mathcal{P}_{k}, \bm{0}, \bm{y})$, we first map 
$\bm{x} \mapsto (\bm{x},0)$, and then select an existing lattice point 
$\bm{y} \in \mathbb{Z}^{k+1} \cap \mathcal{P}_{k} \times \{0\}$, 
which corresponds to a node in $G_k$ (where choosing $\bm{y}=0$ corresponds to selecting the root node). 
This operation introduces exactly two new lattice points to the newly formed polytope: $\bm{v}_{k+1} := \bm{e}_{k+1}$ 
and $\bm{w}_{k+1} := \bm{y}-\bm{e}_{k+1}$. In our graph, we represent this geometric addition by attaching a 
new triangle directly to the node corresponding to $\bm{y}$. The two newly added vertices of this triangle naturally 
represent $\bm{v}_{k+1}$ and $\bm{w}_{k+1}$ and we will later refer to $\bm{y}$ as the parent of $\bm{v}_{k+1}$ and 
$\bm{w}_{k+1}$. This procedure can be visualized in Figure~\ref{fig:full_construction}.

Notice the invariant of this construction: the two newly added vertices always perfectly sum to their parent, 
$\bm{v}_k + \bm{w}_k = \bm{y}$. Before proving the main equivalence theorem, we must establish a lemma regarding 
this additive relationship.

\begin{lemma} \label{lemma:triangles}
For any toric diagram $\mathcal{P} \in \mathcal{C}_n$, three distinct points $\bm{a},\bm{b}, \bm{c} \in \mathcal{P} \cap \mathbb{Z}^n$ satisfy 
$\bm{a} + \bm{b} = \bm{c}$ if and only if the corresponding vertices form a triangle in the associated $3$-cactus $G$, with $\bm{c}$ as the parent 
of the others.
\end{lemma}

\begin{proof} 

\noindent $(\impliedby)$ This follows directly from the construction.
		
\noindent $(\implies)$ Suppose $\bm{a} + \bm{b} = \bm{c}$ for three distinct lattice points in $\mathcal{P}$. Because they are distinct, 
neither $\bm{a}$ nor $\bm{b}$ can be the origin $\bm{0}$ (otherwise we would have $\bm{a} = \bm{c}$ or $\bm{b} = \bm{c}$).

Let $k \le n$ be the largest coordinate index where at least one of these vertices has a non-zero entry. All coordinates in our 
construction belong to $\{-1, 0, 1\}$.
		
If $c_k \neq 0$, assume without loss of generality $c_k = 1$. To satisfy $a_k + b_k = 1$ with coordinates restricted to $\{-1, 0, 1\}$, 
exactly one of the summands, say $a_k$, must equal $1$. By the pseudo-bipyramid construction, there is a unique point in 
$\mathcal{P}$ with a $k$-th coordinate of $1$ and all subsequent coordinates $0$. This forces $\bm{a} = \bm{c}$, which contradicts 
the assumption that $\bm{a}$ and $\bm{c}$ are distinct points.
		
Thus, we must have $c_k = 0$, forcing $\{a_k, b_k\} = \{1, -1\}$. This uniquely identifies $\bm{a}$ and $\bm{b}$ as the exact pair of 
vertices added at dimension $k$. By Theorem~\ref{thm:pbipyr}, these are $(\bm{0}, 1)$ and $(\bm{y}, -1)$. Their sum is $(\bm{y}, 0)$, 
identifying $\bm{c}$ as their parent vertex embedded in the $(k-1)$-dimensional base.
\end{proof}

	\begin{figure}[H]
		\centering
		\begin{tikzpicture}[
			every label/.append style={font=\scriptsize}, 
			root/.style={circle, draw, thick, fill=black, inner sep=1.5pt}, 
			leafcontext/.style={circle, draw=gray, thick, fill=white, inner sep=1.5pt},
			parentactive/.style={circle, draw=violet, thick, fill=violet!30, inner sep=1.5pt},
			nodesnew/.style={circle, draw=cyan, thick, fill=cyan!10, inner sep=1.5pt},
			scale=0.9,
			>=stealth
			]
			\def\L{0.9} 
			
			\node[font=\bfseries] at (0, 2.0) {Polytope};
			\node[font=\bfseries] at (5.5, 2.0) {$3$-Cactus};
			
			
			\begin{scope}[shift={(0, 0)}]
				\draw[->, gray!50] (-2.2, 0) -- (2.2, 0) node[right, text=black, font=\scriptsize] {$x_1$};
				\draw[thick] (-1.2, 0) -- (1.2, 0);
				
				\node[root, label=above:{$0$}] at (0, 0) {};
				\node[leafcontext, label=below:{$-1$}] at (-1.2, 0) {};
				\node[leafcontext, label=below:{$1$}] at (1.2, 0) {};
			\end{scope}
			
			\begin{scope}[shift={(5.5, 0.4)}]
				\coordinate (R1) at (0,0);
				\coordinate (L1) at ($(R1) + (240:\L)$);
				\coordinate (L2) at ($(R1) + (300:\L)$);
				
				\draw[thick, fill=gray!10] (R1) -- (L1) -- (L2) -- cycle;
				
				\node[root, label=above:{$0$}] at (R1) {};
				\node[leafcontext, label=below left:{$-1$}] at (L1) {};
				\node[leafcontext, label=below right:{$1$}] at (L2) {};
			\end{scope}
			
			\draw[->, thick, dashed, black!80] (2.75, -1.0) -- (2.75, -3.5) 
			node[midway, fill=white, align=center, font=\small] {Attach $e_2$ and $y_1-e_2$\\to $y_1=(-1,0)$};
			
			
			\begin{scope}[shift={(0, -5.5)}]
				\draw[->, gray!50] (-2.2,0) -- (2.2,0) node[right, text=black, font=\scriptsize] {$x_1$};
				\draw[->, gray!50] (0,-2) -- (0,2) node[above, text=black, font=\scriptsize] {$x_2$};
				
				\coordinate (V1) at (1.2,0);    
				\coordinate (V2) at (0,1.2);    
				\coordinate (V3) at (-1.2,0);   
				\coordinate (V4) at (-1.2,-1.2);
				\coordinate (Origin) at (0,0);
				
				\draw[thick, fill=gray!10] (V1) -- (V2) -- (V3) -- (V4) -- cycle;
				
				\node[root, label=below right:{$0$}] at (Origin) {};
				
				\node[leafcontext] at (V1) {};
				
				\node[parentactive, label=left:{$y_1=(-1,0)$}] at (V3) {};
				\node[nodesnew, label=above:{$e_2=(0,1)$}] at (V2) {};
				\node[nodesnew, label=below left:{$y_1-e_2=(-1,-1)$}] at (V4) {};
			\end{scope}
			
			\begin{scope}[shift={(5.5, -4.6)}]
				\coordinate (R2) at (0,0);
				\coordinate (L2A) at ($(R2) + (240:\L)$);
				\coordinate (L2B) at ($(R2) + (300:\L)$);
				
				\coordinate (N1) at ($(L2A) + (210:\L)$);
				\coordinate (N2) at ($(L2A) + (270:\L)$);
				
				\draw[thick, fill=gray!10] (R2) -- (L2A) -- (L2B) -- cycle;
				\draw[thick, draw=cyan, fill=cyan!10] (L2A) -- (N1) -- (N2) -- cycle;
				
				\node[root, label=above:{$0$}] at (R2) {};
				
				\node[leafcontext] at (L2B) {};
				
				\node[parentactive, label=above left:{$y_1=(-1,0)$}] at (L2A) {};
				\node[nodesnew, label=below left:{$y_1-e_2=(-1,-1)$}] at (N1) {};
				\node[nodesnew, label=right:{$e_2=(0,1)$}] at (N2) {};
			\end{scope}
			
			\draw[->, thick, dashed, black!80] (2.75, -7.0) -- (2.75, -9.5) 
			node[midway, fill=white, align=center, font=\small] {Attach $e_3$ and $y_2-e_3$\\to $y_2=(0,1,0)$};
			
			
			\begin{scope}[shift={(0, -11.5)}, x={(-0.5cm,-0.3cm)}, y={(1cm,0cm)}, z={(0cm,1cm)}]
				\draw[->, gray!50] (-2,0,0) -- (2,0,0) node[below left, text=black, font=\scriptsize] {$x_1$};
				\draw[->, gray!50] (0,-1.5,0) -- (0,2,0) node[right, text=black, font=\scriptsize] {$x_2$};
				\draw[->, gray!50] (0,0,-1.5) -- (0,0,2) node[above, text=black, font=\scriptsize] {$x_3$};
				
				\coordinate (B1) at (1.2, 0, 0); 
				\coordinate (B2) at (0, 1.2, 0); 
				\coordinate (B3) at (-1.2, 0, 0);
				\coordinate (B4) at (-1.2, -1.2, 0);
				
				\coordinate (T1) at (0, 0, 1.2);       
				\coordinate (T2) at (0, 1.2, -1.2);    
				\coordinate (Org) at (0,0,0);
				
				\draw[fill=gray!10, draw=none] (T1) -- (B4) -- (B1) -- (T2) -- (B2) -- cycle;
				
				\draw[thick, dashed] (T1) -- (B3);
				\draw[thick, dashed] (B4) -- (B3);
				\draw[thick, dashed] (B3) -- (B2);
				\draw[thick, dashed] (T2) -- (B3);
				\draw[thick, dashed] (T2) -- (B4);
				
				\draw[thick] (T1) -- (B4);
				\draw[thick] (B4) -- (B1);
				\draw[thick] (B1) -- (T2);
				\draw[thick] (T2) -- (B2);
				\draw[thick] (B2) -- (T1);
				\draw[thick] (T1) -- (B1);
				\draw[thick] (B1) -- (B2);
				
				\node[root, label=below right:{$0$}] at (Org) {};
				
				\node[leafcontext] at (B1) {}; 
				\node[leafcontext] at (B3) {}; 
				\node[leafcontext] at (B4) {};
				
				\node[parentactive, label=below right:{$y_2=(0,1,0)$}] at (B2) {};
				\node[nodesnew, label=above:{$e_3=(0,0,1)$}] at (T1) {};
				\node[nodesnew, label=below left:{$y_2-e_3=(0,1,-1)$}] at (T2) {};
			\end{scope}
			
			\begin{scope}[shift={(5.5, -10.2)}]
				\coordinate (R3) at (0,0);
				\coordinate (L3A) at ($(R3) + (240:\L)$);
				\coordinate (L3B) at ($(R3) + (300:\L)$);
				
				\coordinate (N3A) at ($(L3A) + (210:\L)$);
				\coordinate (N3B) at ($(L3A) + (270:\L)$); 
				
				\coordinate (P3A) at ($(N3B) + (240:\L)$);
				\coordinate (P3B) at ($(N3B) + (300:\L)$);
				
				\draw[thick, fill=gray!10] (R3) -- (L3A) -- (L3B) -- cycle;
				\draw[thick, fill=white] (L3A) -- (N3A) -- (N3B) -- cycle;
				\draw[thick, draw=cyan, fill=cyan!10] (N3B) -- (P3A) -- (P3B) -- cycle;
				
				\node[root, label=above:{$0$}] at (R3) {};
				
				\node[leafcontext] at (L3B) {}; 
				\node[leafcontext] at (L3A) {};
				\node[leafcontext] at (N3A) {};
				
				\node[parentactive, label=right:{$y_2=(0,1,0)$}] at (N3B) {};
				\node[nodesnew, label=below left:{$e_3=(0,0,1)$}] at (P3A) {};
				\node[nodesnew, label=below right:{$y_2-e_3=(0,1,-1)$}] at (P3B) {};
			\end{scope}
			
		\end{tikzpicture}
		\caption{Visualizing the iterative construction sequence. At each dimension step $k$, only the active parent vertex $\bm{y}_{k-1}$ (Purple) 
		and the two newly generated integer vertices $\bm{e}_k$ and $\bm{y}_{k-1}-\bm{e}_k$ (Cyan) are labeled, explicitly isolating the newly 
		formed triangle.}
		\label{fig:full_construction}
	\end{figure}
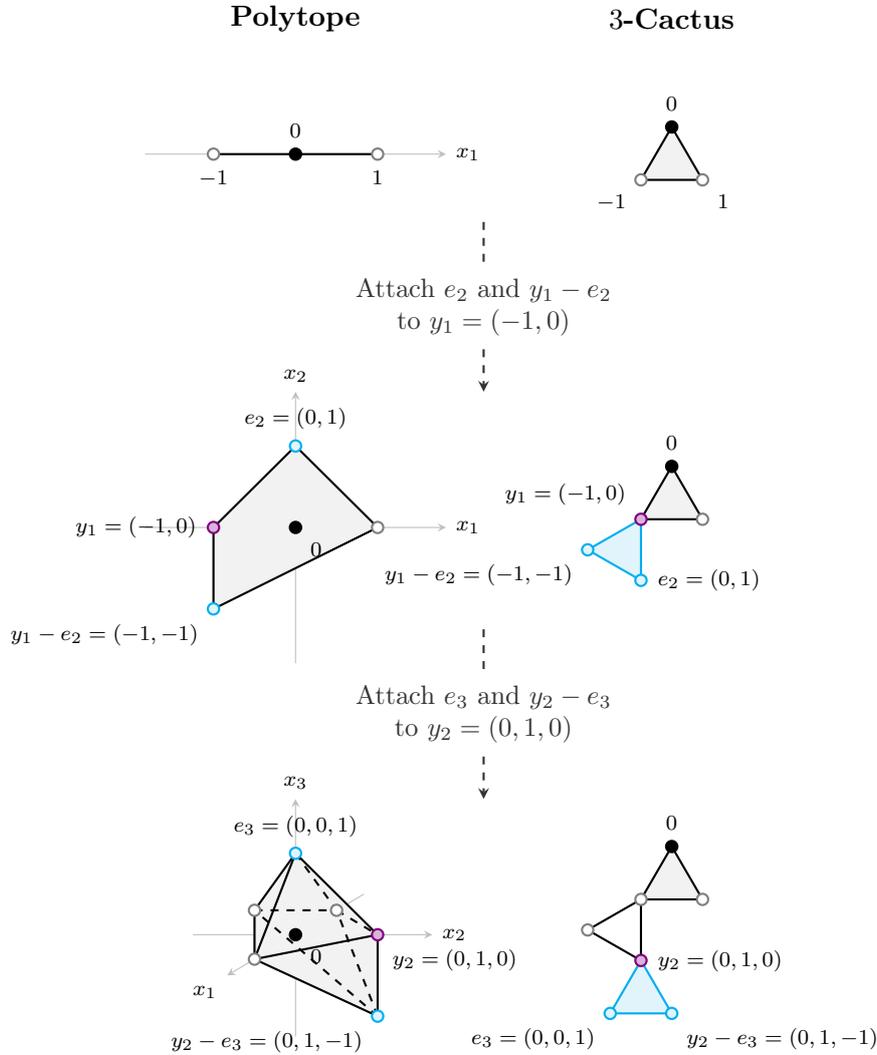

We define a leaf triangle in a rooted $3$-cactus $G$ as a triangle containing exactly two vertices of degree $2$ (the leaves of the graph), 
with the third vertex acting as their parent.
\begin{lemma}\label{lemma:permutation}
The unimodular equivalence class of a polytope $\mathcal{P} \in \mathcal{C}_n$ depends only on its associated $3$-cactus $G$, 
and not on the specific sequence in which its triangles are constructed. Consequently, any leaf triangle of $G$ can be assumed to 
correspond to the final construction step in dimension $n$.
\end{lemma}

\begin{proof}
The rooted structure of $G$ imposes a natural partial order on its triangles, as any parent vertex must be generated before its children. 
A valid geometric construction of $\mathcal{P}$ corresponds to choosing a linear extension of this partial order, assigning each of the 
$n$ triangles to a distinct step $k \in \{1, \dots, n\}$.
		
By the pseudo-bipyramid construction, at each step $k \in \{ 2, \dots, n \}$, the operation embeds the $(k-1)$-dimensional base into 
$\mathbb{R}^{k}$ and introduces two new vertices. Specifically, the two vertices added at step $k$ take the values $\pm 1$ exclusively 
in the $k$-th coordinate. As a result, the geometric embedding of each triangle in $G$ is entirely determined by the axis $\bm{e}_{k}$ 
assigned at its creation step.
		
Suppose we generate two polytopes, $\mathcal{P}_1$ and $\mathcal{P}_2$, using two different valid construction sequences for the 
exact same graph $G$. Because the set of triangles and their parent-child adjacencies are identical, the step sequence for 
$\mathcal{P}_2$ is merely a permutation of the sequence for $\mathcal{P}_1$. Geometrically, this simply permutes the assignment 
of the orthogonal basis vectors $\{\bm{e}_1, \dots, \bm{e}_n\}$ to the triangles of $G$. The linear transformation that maps the vertices 
of $\mathcal{P}_1$ to $\mathcal{P}_2$ is exactly this permutation of the coordinate axis. Because permutation matrices have determinant 
$\pm 1$, this transformation is unimodular, therefore, any two valid attachment sequences for $G$ generate polytopes that are unimodularly 
equivalent.
		
This allows us to freely reorder the construction so that any chosen leaf triangle is generated at the final step $n$.
\end{proof}

\begin{theorem} \label{thm:3cactus_iso}
Two polytopes $\mathcal{P}_1, \mathcal{P}_2 \in \mathcal{C}_n$ are unimodularly equivalent if and only if their 
corresponding rooted $3$-cacti $G_1$ and $G_2$ are isomorphic.
\end{theorem}
\begin{proof}
$(\implies)$ Suppose $\mathcal{P}_1$ and $\mathcal{P}_2$ are unimodularly equivalent. By definition, there exists an unimodular 
transformation $U: \mathbb{Z}^n \to \mathbb{Z}^n$ mapping $\Vv(\mathcal{P}_1)$ bijectively to $\Vv(\mathcal{P}_2)$. 
		
All polytopes in $\mathcal{C}_n$ are reflexive and the origin $\bm{0}$ is their unique interior lattice point. For $U$ to map 
$\mathcal{P}_1$ to $\mathcal{P}_2$ bijectively, it must map the interior to the interior, forcing $U(\bm{0}) = \bm{0}$. 
This guarantees that $U$ is linear (possessing no translation component).
		
Because $U$ is a linear bijection, the additive relation $\bm{a} + \bm{b} = \bm{c}$ holds for a triplet of vertices in 
$\mathcal{P}_1$ if and only if $U(\bm{a}) + U(\bm{b}) = U(\bm{c})$ holds in $\mathcal{P}_2$. 
By Lemma~\ref{lemma:triangles}, these sums identify the adjacencies of the $3$-cactus.
		
We claim the restriction $U|_{\Vv(\mathcal{P}_1)}$ acts as a strict, root-preserving graph isomorphism between $G_1$ and $G_2$.
First, since $U(\bm{0}) = \bm{0}$, the root of the graph is preserved. Second, two distinct vertices $\bm{x}, \bm{y} \in \Vv (G_1)$ 
share an edge if and only if they belong to the same triangle. By Lemma~\ref{lemma:triangles}, this occurs if and only if there exists 
a third vertex $\bm{z} \in \Vv (G_1)$ such that one of the three vertices is the exact sum of the other two. Because $U$ preserves 
this exact additive relation, this holds if and only if the same equality holds in $\mathcal{P}_2$. This sum guarantees that 
$\{U(\bm{x}), U(\bm{y}), U(\bm{z})\}$ forms a valid triangle in $G_2$. Therefore, $\{\bm{x}, \bm{y}\}$ is an edge in $G_1$ if and only 
if $\{U(\bm{x}), U(\bm{y})\}$ is an edge in $G_2$, proving the isomorphism.
		
\noindent $(\impliedby)$ Suppose $G_1$ and $G_2$ are isomorphic via a root-preserving graph isomorphism 
$\phi: \Vv (G_1) \to \Vv (G_2)$. We will explicitly construct a unimodular matrix $U$ such that 
$U(\bm{x}) = \phi(\bm{x})$ for all $\bm{x} \in \Vv (\mathcal{P}_1)$, proceeding by induction on $n$.
		
For the base case $n=1$ we have that $G_1$ and $G_2$ are single triangles corresponding to $[-1, 1]$. The unimodular map 
$U = [1]$ or $U = [-1]$ trivially satisfies the claim depending on whether the isomorphism $\phi$ fixes or swaps the vertices.
		
For the inductive step let $n>1$ and assume the claim holds for dimension $n-1$. Because $G_1$ is a finite $3$-cactus, 
it contains at least one leaf triangle $\{\bm{y}, \bm{v}, \bm{w}\}$ where $\bm{y}$ is the parent vertex. 
By Lemma~\ref{lemma:permutation}, we may assume up to a unimodular permutation that this leaf triangle was added at the final 
$n$-th step. Thus, $\bm{v} = \bm{e}_n$ and $\bm{w} = \bm{y}-\bm{e}_n$, producing the relation $\bm{v} + \bm{w} = \bm{y}$. 
The isomorphism $\phi$ maps this to a leaf triangle $\{\phi(\bm{y}), \phi(\bm{v}), \phi(\bm{w})\}$ in $G_2$. 
Applying Lemma~\ref{lemma:permutation} to $\mathcal{P}_2$, we can reindex its construction steps so this corresponding 
mapped leaf is also added at the $n$-th step. Because it is attached to the parent $\phi(\bm{y})$, the newly added vertices 
$\{\phi(\bm{v}), \phi(\bm{w})\}$ are exactly $\bm{e}_d$ and $\phi(\bm{y}) -\bm{e}_n$. Consequently, $\phi(\bm{v})$ has $n$-th 
coordinate $\alpha \in \{1, -1\}$ and its first $n-1$ coordinates form some vector $\bm{q} \in \mathbb{Z}^{n-1}$. 
Removing these leaf vertices from $G_1$ and $G_2$ leaves isomorphic $3$-cacti of dimension $n-1$. By the inductive hypothesis, 
there exists a $(n-1) \times (n-1)$ unimodular matrix $U'$ mapping all remaining base vertices perfectly, ensuring $U'(\bm{y})  = \phi(\bm{y})$. 
We extend $U'$ to a $n \times n$ matrix $U$ by defining the action on the new basis vector as $U(\bm{e}_n) = \phi(\bm{v})$. 
Because all base vectors in $\mathcal{P}_{n-1}$ have a $0$ in the $n$-th coordinate, $U$ naturally forms a block matrix given by
\[
U = \begin{bmatrix} 
U' & q \\ 
\mathbf{0}^T & \alpha
\end{bmatrix}.
\]
Its determinant is $\det(U) = \alpha \det(U') = \pm 1$, meaning $U$ is a unimodular map. Finally, we verify that $U$ maps the 
remaining leaf vertex $\bm{w}$ correctly. By linearity and the additive relation of the triangle in $G_1$,
\[
U(\bm{w}) = U(\bm{y} - \bm{v}) = U(\bm{y}) - U(\bm{v}) = \phi(\bm{y}) - \phi(\bm{v})\,.
\]
Because $\{\phi(\bm{y}), \phi(\bm{v}), \phi(\bm{w})\}$ is a triangle in $G_2$, we know from Lemma ~\ref{lemma:triangles} that 
$\phi(\bm{v}) + \phi(\bm{w}) = \phi(\bm{y})$. Substituting this directly gives $U(\bm{w}) = \phi(\bm{w})$. Since $U$ matches 
$\phi$ on the base vertices and correctly maps both new leaves, the polytopes are unimodularly equivalent.
\end{proof}

	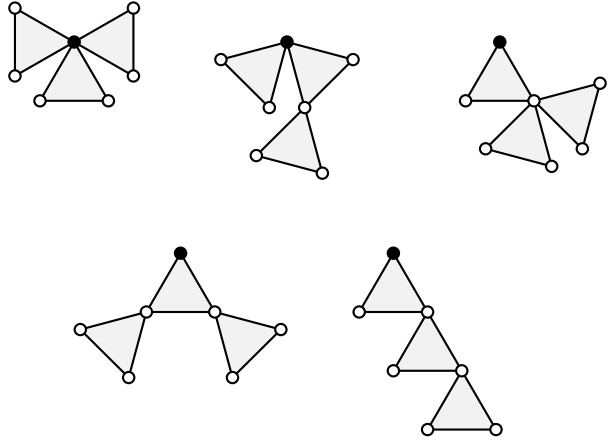
\begin{figure}[H]
		\centering
		\begin{tikzpicture}[
			root/.style={circle, draw, thick, fill=black, inner sep=1.5pt},
			leaf/.style={circle, draw, thick, fill=white, inner sep=1.5pt},
			scale=1.0
			]
			\def\L{0.9} 
			
			\begin{scope}[shift={(0,0)}]
				\coordinate (R1) at (0,0);
				\coordinate (G1T1A) at ($(R1) + (150:\L)$);
				\coordinate (G1T1B) at ($(R1) + (210:\L)$);
				\coordinate (G1T2A) at ($(R1) + (240:\L)$);
				\coordinate (G1T2B) at ($(R1) + (300:\L)$);
				\coordinate (G1T3A) at ($(R1) + (330:\L)$);
				\coordinate (G1T3B) at ($(R1) + (30:\L)$);
				
				\draw[thick, fill=gray!10] (R1) -- (G1T1A) -- (G1T1B) -- cycle;
				\draw[thick, fill=gray!10] (R1) -- (G1T2A) -- (G1T2B) -- cycle;
				\draw[thick, fill=gray!10] (R1) -- (G1T3A) -- (G1T3B) -- cycle;
				
				\node[root] at (R1) {};
				\node[leaf] at (G1T1A) {}; \node[leaf] at (G1T1B) {};
				\node[leaf] at (G1T2A) {}; \node[leaf] at (G1T2B) {};
				\node[leaf] at (G1T3A) {}; \node[leaf] at (G1T3B) {};
			\end{scope}
			
			\begin{scope}[shift={(2.8,0)}]
				\coordinate (R2) at (0,0);
				\coordinate (G2T1A) at ($(R2) + (195:\L)$);
				\coordinate (G2T1B) at ($(R2) + (255:\L)$);
				\coordinate (G2T2A) at ($(R2) + (285:\L)$);
				\coordinate (G2T2B) at ($(R2) + (345:\L)$);
				\coordinate (G2T3A) at ($(G2T2A) + (225:\L)$);
				\coordinate (G2T3B) at ($(G2T2A) + (285:\L)$);
				
				\draw[thick, fill=gray!10] (R2) -- (G2T1A) -- (G2T1B) -- cycle;
				\draw[thick, fill=gray!10] (R2) -- (G2T2A) -- (G2T2B) -- cycle;
				\draw[thick, fill=gray!10] (G2T2A) -- (G2T3A) -- (G2T3B) -- cycle;
				
				\node[root] at (R2) {};
				\node[leaf] at (G2T1A) {}; \node[leaf] at (G2T1B) {};
				\node[leaf] at (G2T2A) {}; \node[leaf] at (G2T2B) {};
				\node[leaf] at (G2T3A) {}; \node[leaf] at (G2T3B) {};
			\end{scope}
			
			\begin{scope}[shift={(5.6,0)}]
				\coordinate (R3) at (0,0);
				\coordinate (G3T1A) at ($(R3) + (240:\L)$);
				\coordinate (G3T1B) at ($(R3) + (300:\L)$);
				\coordinate (G3T2A) at ($(G3T1B) + (225:\L)$);
				\coordinate (G3T2B) at ($(G3T1B) + (285:\L)$);
				\coordinate (G3T3A) at ($(G3T1B) + (315:\L)$);
				\coordinate (G3T3B) at ($(G3T1B) + (15:\L)$);
				
				\draw[thick, fill=gray!10] (R3) -- (G3T1A) -- (G3T1B) -- cycle;
				\draw[thick, fill=gray!10] (G3T1B) -- (G3T2A) -- (G3T2B) -- cycle;
				\draw[thick, fill=gray!10] (G3T1B) -- (G3T3A) -- (G3T3B) -- cycle;
				
				\node[root] at (R3) {};
				\node[leaf] at (G3T1A) {}; \node[leaf] at (G3T1B) {};
				\node[leaf] at (G3T2A) {}; \node[leaf] at (G3T2B) {};
				\node[leaf] at (G3T3A) {}; \node[leaf] at (G3T3B) {};
			\end{scope}
			
			\begin{scope}[shift={(1.4,-2.8)}]
				\coordinate (R4) at (0,0);
				\coordinate (G4T1A) at ($(R4) + (240:\L)$);
				\coordinate (G4T1B) at ($(R4) + (300:\L)$);
				\coordinate (G4T2A) at ($(G4T1A) + (195:\L)$);
				\coordinate (G4T2B) at ($(G4T1A) + (255:\L)$);
				\coordinate (G4T3A) at ($(G4T1B) + (285:\L)$);
				\coordinate (G4T3B) at ($(G4T1B) + (345:\L)$);
				
				\draw[thick, fill=gray!10] (R4) -- (G4T1A) -- (G4T1B) -- cycle;
				\draw[thick, fill=gray!10] (G4T1A) -- (G4T2A) -- (G4T2B) -- cycle;
				\draw[thick, fill=gray!10] (G4T1B) -- (G4T3A) -- (G4T3B) -- cycle;
				
				\node[root] at (R4) {};
				\node[leaf] at (G4T1A) {}; \node[leaf] at (G4T1B) {};
				\node[leaf] at (G4T2A) {}; \node[leaf] at (G4T2B) {};
				\node[leaf] at (G4T3A) {}; \node[leaf] at (G4T3B) {};
			\end{scope}
			
			\begin{scope}[shift={(4.2,-2.8)}]
				\coordinate (R5) at (0,0);
				\coordinate (G5T1A) at ($(R5) + (240:\L)$);
				\coordinate (G5T1B) at ($(R5) + (300:\L)$);
				\coordinate (G5T2A) at ($(G5T1B) + (240:\L)$);
				\coordinate (G5T2B) at ($(G5T1B) + (300:\L)$);
				\coordinate (G5T3A) at ($(G5T2B) + (240:\L)$);
				\coordinate (G5T3B) at ($(G5T2B) + (300:\L)$);
				
				\draw[thick, fill=gray!10] (R5) -- (G5T1A) -- (G5T1B) -- cycle;
				\draw[thick, fill=gray!10] (G5T1B) -- (G5T2A) -- (G5T2B) -- cycle;
				\draw[thick, fill=gray!10] (G5T2B) -- (G5T3A) -- (G5T3B) -- cycle;
				
				\node[root] at (R5) {};
				\node[leaf] at (G5T1A) {}; \node[leaf] at (G5T1B) {};
				\node[leaf] at (G5T2A) {}; \node[leaf] at (G5T2B) {};
				\node[leaf] at (G5T3A) {}; \node[leaf] at (G5T3B) {};
			\end{scope}
			
		\end{tikzpicture}
		\caption{The 5 unimodular equivalence classes of rooted $3$-cacti for $n=3$.}
		\label{fig:3cacti_d3}
	\end{figure}

By viewing these polytopes purely as rooted $3$-cacti, we can bypass geometric checks and efficiently 
calculate the number of unique isomorphism classes. This allows us to compute the values in 
Table~\ref{table:resultados} as well as vertex descriptions of all those toric diagrams. We remark 
that this sequence is known as Entry A003080 in~\cite{OEIS}.
	
	\begin{table}[H]
		\centering
		\begin{tabular}{||p{5cm} p{5cm}||} 
			\hline
			Dimension & Count of P-bipyramids Ehrhart-equivalent to cross-polytopes \\ [0.5ex] 
			\hline\hline
			1 & 1  \\ 
			2 & 2 \\
			3 & 5 \\
			4 & 13 \\
			5 & 37 \\
			6 & 111 \\
			7 & 345 \\
			8 & 1105 \\
			9 & 3624 \\
			10 & 12099 \\
			11 & 41000 \\
			12 & 140647 \\
			13 & 487440 \\
			14 & 1704115 \\ 
			15 & 6002600 \\ [1ex] 
			\hline
		\end{tabular}
		\caption{Number of pseudo-bipyramids Ehrhart equivalent to cross-polytopes for $n \leq 15$}
		\label{table:resultados}
	\end{table}

\section{Toric diagrams Ehrhart equivalent to a small bipyramid over the unimodular simplex} \label{sec:bipysimplex}

Let $n > 1$ be an integer and $\triangle_{n-1}$ denote the standard unimodular simplex in $\mathbb{R}^{n-1}$.
Consider the family of $n$-dimensional polytopes $\mathcal{P}_k$ given by
\[
\mathcal{P}_k := \text{conv}\left( \{(n+k)\triangle_{n-1} - \bm{1}\} \times \{-1\}, \{(n-k)\triangle_{n-1} - \bm{1}\}  \times \{1\} \right)\,,
\]
where $0 \leq k < n$ is an integer parameter. It is easy to see that $\mathcal{P}_k$ is reflexive and Delzant by 
looking at the hyperplane description of $\mathcal{P}_k$, which is given by:
\begin{itemize}
\item $x_i +1 \geq 0, \forall i \in \{1, \dots, n \}$;
\item $- x_n + 1 \geq 0$;
\item $-\sum_{i=1}^{n-1} x_i- k x_n  +1 \geq 0$.
\end{itemize}
The toric diagram $\mathcal{T}_k$ of the prequantization of $\mathcal{P}_k$ has vertex set
\[
\Vv (\mathcal{T}_k) = \{ \bm{e}_1, \dots, \bm{e}_n, -\bm{e}_n, \bm{v}_k \} 
\]
where $\bm{v}_k := (-1, \dots, -1, -k)$.
\begin{lemma} \label{lemma:dependencyreflexive}
For any $0 \le k < n$, the toric diagram $\mathcal{T}_k$ is a bipyramid with apices $\Aa = \{\bm{e}_n, -\bm{e}_n\}$ and 
base $\Ee = \{\bm{e}_1, \dots, \bm{e}_{n-1}, \bm{v}_k\}$.
\end{lemma}
\begin{proof}
Let $B(\mathcal{T}_k)$ be the barycenter of $\Ee$. We have the following affine dependency:
\[
B(\mathcal{T}_k) := \frac{1}{n} \sum_{\bm{v} \in \Ee} \bm{v} = 
\left(\frac{n-k}{2n}\right) \bm{e}_n + \left(\frac{n+k}{2n}\right) (-\bm{e}_n)
\]
Since $0 \le k < n$, the coefficients are positive and sum to $1$. 
By the following Lemma~\ref{lemma:radonpartition}, $\mathcal{T}_k$ is a bipyramid.
\end{proof}
\begin{lemma} \label{lemma:radonpartition}
Let $\Vv = \{\bm{v}_1, \dots, \bm{v}_n, \bm{y}_1, \bm{y}_2\}$ be the set of $n+2$ vertices of a $n$-dimensional lattice polytope 
$\Pp$. If there exists an affine dependency of the form
\[
\sum_{i=1}^n \alpha_i \bm{v}_i = \mu_1 \bm{y}_1 + \mu_2 \bm{y}_2 \,,
\]
where $\alpha_i, \mu_j > 0$ for all $i, j$ and $\sum \alpha_i = \sum \mu_j = 1$, then $\Pp$ is a lattice bipyramid with base 
$\Ee = \{\bm{v}_1, \dots, \bm{v}_d\}$ and apices $\Aa = \{\bm{y}_1, \bm{y}_2\}$.
\end{lemma}
\begin{proof}
The point $\bm{x} := \sum_{i=1}^n \alpha_i \bm{v}_i = \mu_1 \bm{y}_1 + \mu_2 \bm{y}_2$ lies in the relative interior of both 
$\text{conv}(\Ee)$ and $\text{conv}(\Aa)$. In a $n$-polytope with $n+2$ vertices, the space of affine dependencies is exactly 
one-dimensional. This unique dependency implies that $\Ee$ and $\Aa$ are the only subsets of vertices whose convex hulls
intersect in the interior. A subset of vertices $\Ff \subset \Vv$ forms a face of $\Pp$ if and only if its convex hull does not 
intersect the interior of $\Pp$. Therefore, $\Ff$ is a face if and only if it contains neither $\Ee$ nor $\Aa$. The maximal such 
sets (facets) are obtained by removing exactly one vertex from $\Ee$ and one from $\Aa$. Specifically, the facets of $\Pp$ 
are of the form $\text{conv}((\Ee \setminus \{\bm{v}_i\}) \cup \{\bm{y}_j\})$ for $i \in \{1, \dots, n\}$ and $j \in \{1, 2\}$. 
This is precisely the combinatorial description of a bipyramid with apices $\bm{y}_1$ and $\bm{y}_2$ over a $(n-1)$-simplex 
base with vertices $\bm{v}_1, \dots, \bm{v}_n$.
\end{proof}
\begin{theorem}
The $h^*$-polynomial of $\mathcal{T}_k$ is $h_{\mathcal{T}_k}^*(z) = 1 + 2z + 2z^2 + \dots + 2z^{n-1} + z^n$.
\end{theorem}
\begin{proof}
Since $\mathcal{T}_k$ is a toric diagram we know there is a unimodular triangulation $T$ for $\partial \mathcal{T}_k$. 
Furthermore $\mathcal{T}_k$ is reflexive so taking the convex hull of the origin with each simplex of the boundary 
$\partial \mathcal{T}_k$ induces a unimodular triangulation $K$ of $\mathcal{T}_k$ (see e.g. Exercise 10.3 of~\cite{BR}).
Furthermore, we can use Theorem 10.3 of~\cite{BR} to conclude that 
\[
h_{\mathcal{T}_k}^\ast (z) = h_{\partial \mathcal{T}_k}(z) \,,
\]
where $h_{\partial \mathcal{T}_k}$ is the $h$-polynomial of the boundary, i.e.
\begin{equation} \label{eq:hpol}
h_{\partial \mathcal{T}_k}(z) = \sum_{j=0}^{n} f_{j-1} z^j (1-z)^{n-j} 
\end{equation}
with $f_{j-1} =$ number of $(j-1)$-dimensional faces of $\partial \mathcal{T}_k$ ($f_{-1} = 0$).
		
To compute $h_{\partial \mathcal{T}_k}(z)$, we determine the $f$-vector of the boundary complex. As established, 
$\mathcal{T}_k$ is combinatorially a bipyramid over a $(n-1)$-simplex. For $0 \le j < n$, a $(j-1)$-face is either a 
$(j-1)$-face of the base $(n-1)$-simplex (yielding $\binom{n}{j}$ faces) or the convex hull of a $(j-2)$-face of the base 
with one of the two apices (yielding $2\binom{n}{j-1}$ faces). The facets of $\mathcal{T}_k$ ($j=n$) are exactly the 
convex hulls of the $n$ faces of dimension $(n-2)$ of the base with each of the two apices, yielding $f_{n-1} = 2n$. 
This allows us to express the face numbers for all $0 \le j \le n$ as
\[
f_{j-1} = \binom{n}{j} + 2\binom{n}{j-1} - \delta_{j,n} \,,
\]
where $\delta_{j,n}$ is the Kronecker delta. Substituting this explicit $f$-vector in~(\ref{eq:hpol}), we split the summation 
into three parts:
\[
h_{\partial \mathcal{T}_k}(z) = \sum_{j=0}^{n} \binom{n}{j} z^j (1-z)^{n-j} + 2 \sum_{j=1}^{n} \binom{n}{j-1} z^j (1-z)^{n-j} - z^n\,. 
\]
The first sum is the binomial expansion of $(z + (1-z))^n = 1^n = 1$. For the second sum, we have
\begin{align*}
2z \sum_{i=0}^{n-1} \binom{n}{i} z^i (1-z)^{n-1-i} &= \frac{2z}{1-z} \sum_{i=0}^{n-1} \binom{n}{i} z^i (1-z)^{n-i} \\
&= \frac{2z(1-z^n)}{1-z}\\
&= 2z(1 + z + z^2 + \dots + z^{n-1}).
\end{align*}
Recombining all the evaluated terms, we obtain
\[ 
h_{\partial \mathcal{T}_k}(z) = 1 + 2z + 2z^2 + \dots + 2z^{n-1} + z^n\,. 
\]
Because this polynomial depends only on the dimension $n$, the reflexivity, and the combinatorial bipyramid structure, 
it is invariant across all valid choices of $k$. Since $h_{\mathcal{T}_k}^*(z) = h_{\partial \mathcal{T}_k}(z)$, the result 
follows.
\end{proof}

\begin{remark}
We could also have proved this theorem directly from~(\ref{eq:Betti}), since the monotone toric symplectic manifold
$(M^{2n}_k , [\omega] = c_1 (M^{2n}_k))$ determined by the reflexive Delzant polytope $\Pp_k$, 
cf. Subsection~\ref{ssec:primitive2}, can be easily seen to satisfy
\[
\dim H^{2j} (M^{2n}_k; \Q) = \dim H^{2j} (\CP^{n-1} \times \CP^1 ; \Q) = 2 \ \text{for all} \ j = 1\,, \ \ldots\,, n-1\,.
\]
\end{remark}

\begin{theorem}\label{thm:determinedreflexive}
For integers $0 \le k_1, k_2 < n$, $\mathcal{T}_{k_1} \cong \mathcal{T}_{k_2}$ if and only if $k_1 = k_2$.
\end{theorem}
\begin{proof}
The result is trivial for $n=2$ and we will now prove it for $n \ge 3$. Assume there exists an affine unimodular 
transformation $\phi$ mapping $\mathcal{T}_{k_1}$ to $\mathcal{T}_{k_2}$. Since $\phi$ must preserve the 
unique affine dependency from Lemma~\ref{lemma:dependencyreflexive}, it maps the base barycenter 
$B(\mathcal{T}_{k_1})$ to $B(\mathcal{T}_{k_2})$. This requires
\[
\phi \left( \frac{n-k_1}{2n} \bm{e}_n + \frac{n+k_1}{2n} (-\bm{e}_n) \right) = 
\frac{n-k_2}{2n} \bm{e}_n + \frac{n+k_2}{2n} (-\bm{e}_n)\,.
\]
The map $\phi$ must preserve vertex degrees in the $1$-skeleton. The base vertices $\Ee$ have degree $n+1$, 
while the apices $\Aa$ have degree $n$. Thus, $\phi$ maps $\Aa_{k_1}$ to $\Aa_{k_2}$ and 
$\Ee_{k_1}$ to $\Ee_{k_2}$. As $\phi$ restricts to a bijection on the apices $\Aa_{k_1} = \Aa_{k_2} = \{\bm{e}_n, -\bm{e}_n\}$, 
there are two cases. If $\phi$ preserves the apices, then $(n+k_1)/2n = (n+k_2)/2n$, implying $k_1 = k_2$. If $\phi$ 
swaps the apices, then $(n-k_1)/2n = (n+k_2)/2n$, implying $-k_1 = k_2$, which for $k_1, k_2 \ge 0$ forces 
$k_1 = k_2 = 0$.
\end{proof}

To apply the prequantization to the divided polytope $\mathcal{P}_k / 2$, we must first ensure that the resulting 
scaled polytope is integral. Shifting the polytope by $(1,\dots,1)$, we obtain a hyperplane description of the 
shifted polytope $\mathcal{P}'_k := \mathcal{P}_k + (1,\dots,1)$ which is given by:
\begin{itemize}
\item $x_i \geq 0, \forall i \in \{1, \dots, n\}$;
\item $ -x_n + 2\geq 0$;
\item $-\sum_{i=1}^{n-1} x_i - k x_n +n+k \geq 0$.
\end{itemize}
We conclude that $\tilde\Pp_k := \mathcal{P}'_k / 2$ is integral if and only if $k \equiv n \pmod 2$. The inner normal 
vectors of the cone over $\tilde\Pp_k \times \{1\}\subset\R^{n+1}$ are:
\begin{itemize}
\item $(\bm{e}_i, 0), \forall i \in \{1, \dots, n\}$;
\item $(-\bm{e}_n, 1)$;
\item $(-1, \dots, -1, -k, \frac{n+k}{2})$.
\end{itemize}
We use the unimodular transformation in $\mathbb{R}^{n+1}$ given by the matrix $T \in \text{GL}(n+1, \mathbb{Z})$
\[
T = \begin{bmatrix}
		0 & 1 & 0 & \dots  & 0 \\
		0 & 0 & 1 & \ddots  & \vdots \\
		\vdots & \vdots & \ddots & \ddots & 0 \\
		0 & 0 & \dots & 0 & 1 \\
		1 & 1 & \dots & 1  & 2
\end{bmatrix}
\]
to map all these points to the hyperplane $x_{n+1}=1$. We obtain that the vertices of the corresponding toric diagram 
$\mathcal{D}_k$ are:
\begin{itemize}
\item $\bm{a}_1 = \bm{0}$;
\item $\bm{a}_{i+1} = \bm{e}_i, \forall i \in \{1, \dots, n-1\}$;
\item $\bm{a}_{n+1} = - \bm{e}_{n-1} + \bm{e}_n$;
\item $\bm{a}_{n+2} = - \sum_{i=1}^{n-2} \bm{e}_i - k \bm{e}_{n-1} + \frac{n+k}{2} \bm{e}_n$.
\end{itemize}
See Figure~\ref{fig:bpdiagram} with $\Tt_1$ and $\Dd_1$ in dimension $n=3$.

\begin{figure}[]
\centering
\begin{minipage}{.5\textwidth}
	\centering
	\begin{tikzpicture}[scale=2.0, line join=bevel, tdplot_main_coords]
		\coordinate (e1) at (1,0,0);
		\coordinate (e2) at (0,1,0);
		\coordinate (e3) at (0,0,1);
		\coordinate (m1) at (0,0,-1);
		\coordinate (m2) at (-1,-1,-1);
		
		\draw (e1) -- (e3);
		\draw (e1) -- (e2);
		\draw (e2) -- (e3);
		\draw (e3) -- (m2);
		\draw (e1) -- (m2);
		\draw [dashed] (e2) -- (m2);

		\draw (m1) -- (m2);
		\draw (m1) -- (e1);
		\draw (m1) -- (e2);
				
	\end{tikzpicture}
\end{minipage}%
\begin{minipage}{.5\textwidth}
	\centering
\begin{tikzpicture}[scale=1.5, line join=bevel, tdplot_main_coords]
		\coordinate (e1) at (0,0,0);
		\coordinate (e2) at (0,1,0);
		\coordinate (e3) at (-1,-1,2);
		\coordinate (m1) at (1,0,0);
		\coordinate (m2) at (0,-1,1);
		
		\draw [dashed] (e1) -- (e3);
		\draw [dashed] (e1) -- (e2);
		\draw (e2) -- (e3);
		\draw (e3) -- (m2);
		\draw [dashed] (e1) -- (m2);
		\draw (e2) -- (m2);

		\draw (m1) -- (m2);
		\draw [dashed] (m1) -- (e1);
		\draw (m1) -- (e2);
				
	\end{tikzpicture}
\end{minipage}

\caption{$\Tt_1$ and $\Dd_1$ in dimension $n=3$.}
\label{fig:bpdiagram}
\end{figure}
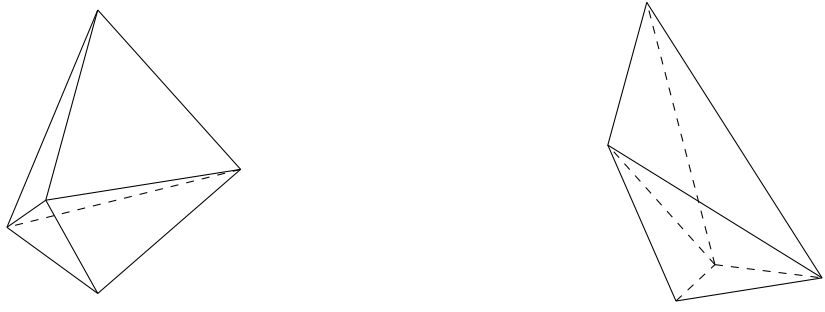

\begin{lemma}\label{lemma:dependencygorenstein}
For any $0 \le k < n$ satisfying $k \equiv n \pmod 2$, the polytope $\mathcal{D}_k$ is a lattice bipyramid 
with apices $\Aa = \{\bm{a}_n, \bm{a}_{n+1}\}$ and base $\Ee =\{\bm{a}_1, \dots, \bm{a}_{n-1}, \bm{a}_{n+2} \}$.
\end{lemma}
\begin{proof}
The polytope $\mathcal{D}_k$ inherits the barycentric relation from the reflexive family $\mathcal{T}_k$:
\[ 
B(\mathcal{D}_k) = \frac{1}{n} \sum_{\bm{v} \in \Ee} \bm{v} = 
\left(\frac{n-k}{2n}\right) \bm{a}_n + \left(\frac{n+k}{2n}\right) \bm{a}_{n+1}\,.
\]
Since $0 \le k < n$, the coefficients are positive and sum to $1$. By Lemma~\ref{lemma:radonpartition}, 
$\mathcal{D}_k$ is a bipyramid.
\end{proof}

\begin{theorem}
The $h^\ast$-polynomial of the toric diagram $\mathcal{D}_k$ is given by 
\[
h^\ast_{\mathcal{D}_k} (z) = 1 + z + z^2 + \dots + z^{n-1}\,.
\]
\end{theorem}
\begin{proof}
From Theorem~\ref{thm:main} we know that 
$h^\ast_{\mathcal{T}_k}=(1+z) h^\ast_{\mathcal{D}_k}$. The result follows.
\end{proof}	

\begin{theorem} \label{thm:uniqueness}
For integers $k_1, k_2$ such that $0 \le k_1, k_2 < n$ and $k_1 \equiv k_2 \equiv n \pmod 2$, the prequantized 
toric diagrams $\mathcal{D}_{k_1}$ and $\mathcal{D}_{k_2}$ are unimodularly equivalent if and only if $k_1 = k_2$.
\end{theorem}
\begin{proof}
The proof is mostly the same as the one for Theorem~\ref{thm:determinedreflexive}. The result is trivial for $n=2$ 
and we will prove it for $n \ge 3$. Assume there exists an affine unimodular transformation $\phi$ mapping 
$\mathcal{D}_{k_1}$ to $\mathcal{D}_{k_2}$. Let $\bm{a}_1, \dots, \bm{a}_{n+2}$ be the vertices of $\mathcal{Q}_{k_1}$ 
and $\bm{a'}_1, \dots, \bm{a'}_{n+2}$ the ones of $\mathcal{Q}_{k_2}$. 
Since $\phi$ must preserve the unique affine dependency from Lemma~\ref{lemma:dependencygorenstein}, it maps 
the base barycenter $B(\mathcal{Q}_{k_1})$ to $B(\mathcal{Q}_{k_2})$. This requires
\[
\phi \left( \frac{n-k_1}{2n} \bm{a}_n + \frac{n+k_1}{2n} \bm{a}_{n+1} \right) 
= \frac{n-k_2}{2n} \bm{a'}_n + \frac{n+k_2}{2n} \bm{a'}_{n+1}\,.
\]
The map $\phi$ must preserve vertex degrees. The base vertices in $\Ee$ have degree $n+1$, while the apices in $\Aa$ 
have degree $n$. Thus, $\phi$ maps $\Aa_{k_1}$ to $\Aa_{k_2}$ and $\Ee_{k_1}$ to $\Ee_{k_2}$. 
Since $\phi$ induces a bijection on the apex sets, the coefficients $\{ \frac{n-k_1}{2n}, \frac{n+k_1}{2n} \}$ must coincide 
with $\{ \frac{n-k_2}{2n}, \frac{n+k_2}{2n} \}$ as multiset entries. The result follows.
\end{proof}

\begin{theorem} \label{thm:bipysimplex2}
Let $\Ss$ be a $n$-dimensional toric diagram. If its $h^*$-polynomial is $h_{\Ss}^\ast (z) = 1 + z + z^2 + \dots + z^{n-1}$, 
then $\Ss$ is unimodularly equivalent to $\mathcal{D}_k$ for a unique $k\in\Z$ satisfying $0 \le k < n$ and 
$k \equiv n \pmod 2$.
\end{theorem}
\begin{proof}
From the given $h^\ast$-polynomial, the normalized volume of $\Ss$ is $h_{\Ss}^\ast(1) = n$. 
The total number of lattice points is $L_{\Ss}(1) = h_1^\ast + n + 1 = n+2$. Since the degree of 
$h_{\Ss}^\ast$ is $n-1$ and it is palindromic, $\Ss$ is Gorenstein of index $2$ thus contains no interior 
lattice points. Because $\Ss$ is a toric diagram, its facets are primitive, so all $n+2$ lattice points are vertices.
		
Next, we determine the number of facets. The standard expansion of the Ehrhart polynomial dictates that the 
coefficient of $t^{n-1}$ in $L_{\Ss}(t)$ is
\[
\frac{1}{2} \sum_{F \in \mathcal{F}(\Ss)} \text{Vol}_{n-1}(F). 
\]
Because every facet $F$ of a toric diagram is a unimodular simplex, the relative volume 
$\text{Vol}_{n-1}(F) = \frac{1}{(n-1)!}$, making this coefficient exactly $\frac{f_{n-1}}{2(n-1)!}$. Using our specific 
$h^\ast$-polynomial, the Ehrhart polynomial is $L_{\Ss}(t) = \sum_{j=0}^{n-1} \binom{t+n-j}{n}$, where the coefficient 
of $t^{n-1}$ is $\frac{n}{(n-1)!}$. Equating these gives a facet count of $f_{n-1} = 2n$.
		
Any $n$-dimensional polytope with $n+2$ vertices admits a unique (up to scaling) affine dependency 
$\sum \lambda_i \bm{v}_i = 0$ with $\sum \lambda_i = 0$. Because every facet of $\Ss$ is a simplex, no subset of 
$n+1$ vertices can be affinely dependent. This strictly guarantees that $\lambda_i \neq 0$ for all $i$, meaning every vertex 
participates in the dependency. This dependency therefore partitions the vertices into exactly two disjoint sets 
$I = \{\bm{v}_i : \lambda_i > 0\}$ and $J = \{\bm{v}_j : \lambda_j < 0\}$. A subset of vertices forms a facet if and only if it is 
obtained by removing exactly one vertex from $I$ and one from $J$. Thus, the number of facets is given by the product 
$|I| \cdot |J|$. We are led to the system
\[ 
|I| + |J| = n+2 \quad \text{and} \quad |I| \cdot |J| = 2n\,,
\]
whose unique integer solution is $\{|I|, |J|\} = \{n, 2\}$. A polytope whose vertices are partitioned in this way is, 
by definition, a bipyramid over a $(n-1)$-simplex. Let $\Ee = \{\bm{v}_1, \dots, \bm{v}_n\}$ be the base vertices and 
$A = \{\bm{y}_1, \bm{y}_2\}$ be the apices.
		
The bipyramid $\Ss$ is triangulated by $n$ lattice simplices $S_i = \text{conv}((\Ee \setminus \{\bm{v}_i\}) \cup \Aa)$,
therefore the sum of their normalized volumes must be $n$. Since each $S_i$ is a lattice simplex, its normalized volume must 
be an integer greater than or equal to $1$. This forces the normalized volume of $S_i$ to be $1$ for all $1 \le i \le n$. The equality 
of these volumes implies that the coefficients of the base vertices in the affine dependency are identical. The dependency relation 
is thus
\[
\sum_{i=1}^n \bm{v}_i = \mu_1 \bm{y}_1 + \mu_2 \bm{y}_2 \,,
\]
where $\mu_1 + \mu_2 = n$. By setting $k = \mu_2 - \mu_1$, we find $\mu_1 = \frac{n-k}{2}$ and $\mu_2 = \frac{n+k}{2}$. 
Up to swapping the apices $\bm{y}_1$ and $\bm{y}_2$, we may assume $0 \le k < n$.
		
Finally, because $\text{Vol}(S_n) = 1$, the vectors $\{\bm{v}_1 - \bm{y}_1, \dots, \bm{v}_{n-1} - \bm{y}_1, \bm{y}_2 - \bm{y}_1\}$ 
form a $\mathbb{Z}$-basis for the lattice. Let $\phi$ be the affine unimodular transformation that maps this basis to the basis 
formed by the corresponding vertices of $\mathcal{D}_k$:
\[ 
\phi(\bm{y}_1) = \bm{a}_n, \quad \phi(\bm{y}_2) = \bm{a}_{n+1}, \quad \phi(\bm{v}_i) = \bm{a}_i \text{ for } 1 \le i \le n-1\,.
\]
Applying $\phi$ to our dependency relation isolates $\phi(\bm{v}_n)$
\[
\sum_{i=1}^{n-1} \bm{a}_i + \phi(\bm{v}_n) = \frac{n-k}{2} \bm{a}_n + \frac{n+k}{2} \bm{a}_{n+1} \implies 
\phi(\bm{v}_n) = \left(-1, \dots, -1, -k, \frac{n+k}{2}\right) \,.
\]
For $\phi(\bm{v}_n)$ to be an integral lattice point, the coordinate $\frac{n+k}{2}$ must be an integer, which forces the condition 
$k \equiv n \pmod 2$. Therefore, $\Ss$ is unimodularly equivalent to $\mathcal{D}_k$. Uniqueness of $k$ is a consequence
of Theorem~\ref{thm:uniqueness}.
\end{proof}


\begin{thebibliography}{aa}

\bibitem{AM0} M.~Abreu and L.~Macarini, {\em Contact homology of good toric contact manifolds.}
Compositio Mathematica {\bf 148} (2012), 304--334.

\bibitem{AM} M.~Abreu and L.~Macarini, {\em On the mean Euler characteristic of Gorenstein toric contact manifolds.}
International Mathematics Research Notices {\bf 14} (2020), 4465-4495.

\bibitem{AMM0} M.~Abreu, L.~Macarini and M~Moreira, {\em On contact invariants of non-simply connected Gorenstein 
toric contact manifolds.} Mathematical Research Letters {\bf 29} (2022), 1--42.

\bibitem{AMM} M.~Abreu, L.~Macarini and M~Moreira, {\em Contact invariants of $\Q$-Gorenstein toric contact manifolds, the Ehrhart polynomial and Chen-Ruan cohomology.} Advances in Mathematics {\bf 429} (2023), 50 pp.

\bibitem{BR} M.~Beck and S.~Robins, {\em Computing the continuous discretely, 2nd edition.} Undergraduate Texts in Mathematics, 
Springer-Verlag, 2015.

\bibitem{BHW} C.~Bey, M.~Henk, and J.~Wills, {\em Notes on the roots of Ehrhart polynomials.} 
Discrete and Computational Geometry {\bf 38} (2007), 81--98.

\bibitem{B} F.~Bourgeois,  {\em A Morse-Bott approach to contact homology.} PhD thesis, Stanford University, 2002.

\bibitem{BCKV} D.~Bump, K.-K.~Choi, P.~Kurlberg, and J.~Vaaler, {\em A local riemann hypothesis, i.} 
Mathematische Zeitschrift {\bf 233} (2000), 1--18.

\bibitem{CLMP-1} Y.~Cho, E.~Lee, M.~Masuda and S.~Park, {\em Unique toric structure on a Fano Bott manifold.}
Journal of Symplectic Geometry {\bf 21} (2023), 439--462.

\bibitem{CLMP-2} Y.~Cho, E.~Lee, M.~Masuda and S.~Park, {\em On the enumeration of Fano Bott manifolds.}
In {\em Toric Topology and Polyhedral Products}, Fields Institute Communications, Springer, 2024.

\bibitem{D} T.~Delzant, {\em Hamiltoniens p\'eriodiques et images convexes de l'application moment.} 
Bulletin de la Soci\'et\'e Math\'ematique de France {\bf 116} (1988), 315--339.

\bibitem{GK} M.~Grossberg and Y.~Karshon, {\em Bott towers, complete integrability, and the extended character 
of representations.} Duke Mathematical Journal {\bf 76} (1994), 23--58.

\bibitem{HNP} C.~Haase, B.~Nill, and A.~Paffenholz, {\em Lecture notes on lattice polytopes.} Draft of June 28, 2021, available online.

\bibitem{HK} A.~Higashitani and K.~Kurimoto,{\em Cohomological rigidity for Fano Bott manifolds.} 
Mathematische Zeitschrift {\bf 301} (2022), 2369--2391.

\bibitem{KPT} P.~Kirschenhofer, A.~Peth\"o and R.~Tichy, {\em On analytical and diophantine properties of a family of counting polynomials.} 
Acta Scientiarum Mathematicarum {\bf 65} (1999), 47--60.

\bibitem{L} E.~Lerman, {\em Contact toric manifolds.} Journal of Symplectic Geometry {\bf 1} (2003), 785--828.

\bibitem{O} M.~Obro, {\em An algorithm for the classification of smooth Fano polytopes.} Preprint. arXiv:0704.0049v1 (2007).

\bibitem{OEIS} The On-Line Encyclopedia of Integer Sequences, https://oeis.org/A003080, OEIS Foundation Inc. (2026).

\bibitem{R-V} F.~Rodriguez-Villegas, {\em On the zeros of certain polynomials.} 
Proceedings of the American Mathematical Society {\bf 130} (2002), 2251--2254.

\bibitem{S} Y.~Suyama, {\em Fano generalized Bott manifolds.} Manuscripta Mathematica {\bf 163} (2020), 427--435.

\bibitem{Z} G.~Ziegler, {\em Lectures on polytopes.} Springer-Verlag, 1995.

\end{thebibliography}
\end{document}